\documentclass[12pt]{amsart}
\usepackage{amsmath,amssymb}

\usepackage[usenames,dvipsnames]{color}
\usepackage{tikz-cd}
\usepackage[margin=1in]{geometry}
\usepackage{hyperref}

\newtheorem{thm}{Theorem}[section]
\newtheorem{lem}[thm]{Lemma}
\newtheorem{prop}[thm]{Proposition}
\newtheorem{cor}[thm]{Corollary}

\newtheorem{deflemma}[thm]{Definition/Lemma}

\theoremstyle{definition}
\newtheorem{defn}[thm]{Definition}
\newtheorem{notn}[thm]{Notation}
\theoremstyle{remark}
\newtheorem{rmk}[thm]{Remark}
\newtheorem{ex}[thm]{Example}

\newcommand{\cat}[1]{\mathbf{#1}} 
\newcommand{\Kappa}{\mathrm{K}}
\newcommand{\Iota}{\mathrm{I}}
\newcommand{\Mu}{\mathrm{M}}

\newcommand{\eps}{\varepsilon}

\newcommand{\R}{\mathbb{R}}
\newcommand{\abs}[1]{\lvert#1\rvert}
\newcommand{\incl}{\hookrightarrow}

\DeclareMathOperator{\colim}{colim}

\begin{document}

% \title[Gromov-Hausdorff Distance and Interleaving]{Categorification of Gromov-Hausdorff Distance and Interleaving of Functors}

\title{Interleaving and Gromov-Hausdorff distance}

\author{Peter Bubenik}
\address[Peter Bubenik]{Department of Mathematics, University of Florida}
\email{peter.bubenik@ufl.edu}

\author{Vin de Silva}
\address[Vin de Silva]{Department of Mathematics, Pomona College}
\email{Vin.deSilva@pomona.edu}

\author{Jonathan Scott}
\address[Jonathan Scott]{Department of Mathematics, Cleveland State University}
\email{j.a.scott3@csuohio.edu}
%\date{\today}

\maketitle

%\doublespacing

\begin{abstract}
One of the central notions to emerge from the study of persistent homology is that of \emph{interleaving distance}. It has found recent applications in symplectic and contact geometry, sheaf theory, computational geometry, and phylogenetics. Here we present a general study of this topic. We define interleaving of  functors with common codomain as solutions to an extension problem. In order to define interleaving distance in this setting we are led to categorical generalizations of Hausdorff distance, Gromov-Hausdorff distance, and the space of metric spaces. We obtain comparisons with previous notions of interleaving via the study of future equivalences. As an application we recover a definition of shift equivalences of discrete dynamical systems. 
\end{abstract}

%\tableofcontents

%\doublespacing

\section{Introduction}

Persistent homology~\cite{elz:tPaS,zomorodianCarlsson:computingPH} has been a highly successful tool in applied topology~\cite{ghrist:survey,carlsson:topologyAndData}. Consequently, it has been a focus of research -- see for example, the recent book~\cite{oudot:book}.
Furthermore, the use of persistent homology and the related interleaving distance in pure and applied mathematics in increasingly sophisticated ways (see Section~\ref{sec:uses} below) has inspired the present work,
which develops a general theory of interleaving and interleaving distance.

\subsection{Persistence modules} 

The central algebraic object of study in this subject is the \emph{persistence module}, which may be viewed as a functor $M$ from $\cat{R}$, the poset $(\R,\leq)$ of real numbers with the usual order considered as a category, to the category $\cat{Vect}$, of vector spaces over a fixed field $K$ and $K$-linear maps. That is, for each real number $a$, we have a vector space $M(a)$ and for each $a \leq b$ we have a linear map $M(a\leq b)$ from $M(a)$ to $M(b)$ such that $M(a \leq a)$ is the identity map and $M(b \leq c) \circ M(a \leq b) = M(a \leq c)$. 
For example, if $f:X \to \R$ is a real-valued function on a topological space, then there is a persistence module given by $F(a) = H_i(f^{-1}(-\infty,a];K)$, where $H_i(-;K)$ is singular homology in degree $i$ with coefficients in the field $K$ and $F(a\leq b)$ is the linear map induced by inclusion.

\subsection{Interleaving and interleaving distance}
\label{sec:interleaving}

Given two such persistence modules, $M,N: \cat{R} \to \cat{Vect}$, how do we compare them? A standard answer is to specify when they are isomorphic: if for all $a\in \R$ we have comparison maps $\varphi(a): M(a) \to N(a)$ and $\psi: N(a) \to M(a)$ such that all the comparison maps and maps given by the persistence modules collectively commute. However, in an applied setting where experimental noise may affect $M$ and $N$ this is hopelessly restrictive. Instead, consider the following more flexible notion. Given $\eps \geq 0$, we say that $M$ and $N$ are \emph{$\eps$-interleaved} if for each $a\in \R$ there exist comparison maps  $\varphi(a): M(a) \to N(a + \eps)$ and $\psi: N(a) \to M(a + \eps)$ such that all the comparison maps and maps given by the persistence modules collectively commute. The infimum of all epsilon for which there exists such an interleaving is called the \emph{interleaving distance}.
This distance was first defined and studied in~\cite{ccsggo:interleaving}.% from a linear algebra point of view. 

\subsection{Uses of interleaving distance}
\label{sec:uses}

One can define generalized persistence modules with indexing categories other than $\cat{R}$ and define interleaving and interleaving distance in this setting~\cite{bdss:1}.
Here we give a brief outline of some recent uses of interleaving distance. Note that there is some overlap between the subjects below.

\subsubsection*{Symplectic geometry}

In~\cite{Polterovich:2016}, the authors apply persistence modules to filtered Floer homology and relate interleaving distance to Hofer's metric.
With these tools they
find robust obstructions to representing a Hamiltonian diffeomorphism as a full $k$-th power and to including it into a one-parameter subgroup.
Also using persistence, in~\cite{Zhang:2016} it is shown the
Hofer distance from a time-dependent Hamiltonian diffeomorphism to the set of $k$-th power Hamiltonian diffeomorphisms can be arbitrarily large in the product structure of any closed symplectic manifold and a closed oriented surface of genus at least four when $k$ is sufficiently large.
In~\cite{Usher:2016}, the authors define persistent homology and interleaving for Novikov's Morse theory for closed one-forms and in Floer theory on not-necessarily monotone symplectic manifolds.
Most recently, in~\cite{Polterovich:2017}, the authors define persistence modules with operators and an associated interleaving distance.
We remark that these persistence modules with operators may be viewed as generalized persistence modules.
% as we will define in this paper (by adding morphisms $a \to a + c$ to $\cat{R}$).

\subsubsection*{Contact geometry}

In~\cite{Alves:2017}, the authors apply interleaving distance to wrapped Floer homology to study its algebraic growth. They use these results to construct a large class of contact manifolds on which all Reeb flows have chaotic dynamics.

\subsubsection*{Sheaf and cosheaf theory}

Sheaves and cosheaves are important examples of generalized persistence modules, with indexing category given by the open sets in some topological space. If this space has a metric, then one can study their interleaving distance.
This has been carefully studied in~\cite[Section 15]{curry:thesis},\cite{Curry:2015},\cite{Curry:2016}, and is also briefly mentioned in~\cite{bdss:1}.

\subsubsection*{Microlocal sheaf theory}

Let $K$ denote a fixed field. In~\cite{Kashiwara:2017}, the authors consider the abelian category of sheaves of $K$-vector spaces on a real vector space and its bounded derived category, from the viewpoint of persistent homology. They define a convolution distance which is an interleaving distance.
Compare with~\cite{curry:thesis}.
Following on this work, in~\cite{Asano:2017}, the authors define an interleaving distance on Tamarkin's category, which is a quotient on the  bounded derived category of sheaves of $K$-vector spaces on $M \times \mathbb{R}$, where $M$ is a connected manifold. They prove that the distance between an object and its Hamiltonian deformation is at most the Hofer norm of the Hamiltonian function.

\subsubsection*{Persistent homology}

Interleaving distance was carefully studied in the context of multiparameter multiparameter persistence modules~\cite{Lesnick:2015}. Follow up work includes \cite{Bauer:2015}, \cite{Botnan:2016} and a recent study of a homotopy-invariant analogue~\cite{Blumberg:2017}.
A categorical perspective of interleaving was taken in~\cite{bubenikScott:1}, which was then considerably generalized in~\cite{bdss:1}. See~\cite{deSilva:2017} for a recent extension of this work.

\subsubsection*{Computational geometry}

Reeb graphs, Reeb spaces, merge trees, the mapper and the multiscale mapper are computational tools for simplifying topological spaces. For example, given a continuous map $f: X \to Y$, the Reeb space is the quotient of $X$ obtained by identifying points in path components of fibers of $f$.
Given a map $f: X \to Y$ and an open cover of $Y$, the mapper is the nerve of the path components of the pullback of the cover.
Interleaving has been used to measure distances between these spaces \cite{Morozov:2013,deSilva:2016,Dey:2016,Carriere:2017,Dey:2017,Bauer:2018,bdss:1}
and to show that the mapper converges to the Reeb space~\cite{Munch:2015}.

\subsubsection*{Phylogenetics}

Phylogenetic trees are edge-weighted trees that represent the evolutionary relationships between biological species.
It has been recently shown that the cophonetic metric on phylogenetic trees~\cite{Cardona2013} can be viewed as an interleaving distance~\cite{Munch:2018}.

\subsection{A new Gromov-Hausdorff approach to interleaving}
\label{sec:new-approach}

To further develop the theory of interleaving and to facilitate its use in a greater variety of settings we present the following more general and conceptually simpler view. 

Given persistence modules $M,N: \cat{R} \to \cat{Vect}$,  an $\eps$-interleaving (as defined above) is exactly a solution to the following extension problem.
\begin{equation*} %\label{eq:eps-interleaving}
	\begin{tikzcd}[row sep=tiny]
		& \cat{I_{\eps}} \arrow[dd,dashed]	\\
		\cat{R} \arrow[ur,hook,"i_1"] \arrow[dr,"M"'] & & \cat{R} \arrow[ul,hook',"i_2"'] \arrow[dl,"N"] \\
		& \cat{Vect}
	\end{tikzcd}
\end{equation*}
Here $\cat{I_{\eps}}$ is the category given by the poset $(\R \amalg \R,\leq_{\eps})$ where if $a$ and $b$ are in the same copy of $\R$ then $a \leq_{\eps} b$ if and only if $a \leq b$ and if $a$ and $b$ are in different copies of $\R$ then $a \leq_{\eps} b$ if and only if $\eps \leq b-a$.
The functors $i_1$ and $i_2$ are given by inclusion of the respective posets.

In fact, this extension problem has a solution if and only if there is an extension to the category $\cat{J_{\eps}}$ given by the larger poset $([0,\eps] \times \R, \leq)$ where the inclusions are given by the two components of the boundary and $(x,a) \leq (y,b)$ if and only if $\abs{x-y} \leq b-a$. That an $\eps$-interleaving can be extended to $\cat{J_{\eps}}$ is known as \emph{interpolation}, and was first proved in~\cite{ccsggo:interleaving} and generalized in~\cite{bdsn}.

We will say that $M$ and $N$ are $\cat{I_{\eps}}$-interleaved and $\cat{J_{\eps}}$-interleaved and that the embeddings of $i_1$ and $i_2$ in both cases have \emph{weight} $\eps$. Thus the interleaving distance between $M$ and $N$ is at most $\eps$.

Now for the general picture. Let $\cat{C}$ be a category, and let $F:\cat{P} \to \cat{C}$ and $G:\cat{Q} \to \cat{C}$ be functors, where $\cat{P}$ and $\cat{Q}$ are small categories. We say that $F$ and $G$ are \emph{$\cat{I}$-interleaved} if there exists a small category $\cat{I}$, embeddings $i: \cat{P} \incl \cat{I}$ and $j: \cat{Q} \incl \cat{I}$, and a functor $H:\cat{I} \to \cat{C}$ making the diagram commute.
\[
	\begin{tikzcd}[row sep=tiny]
		& \cat{I} \arrow[dd,"H"]	\\
		\cat{P} \arrow[ur,hook,"i"] \arrow[dr,"F"'] & & \cat{Q} \arrow[ul,hook',"j"'] \arrow[dl,"G"] \\
		& \cat{C}
	\end{tikzcd}
\]

In order to define interleaving distance, we need to associate a number to such an interleaving. 
We will require $\cat{P}$, $\cat{Q}$ and $\cat{I}$ to be categorical versions of metric spaces (Definition~\ref{defn:weighted-category}). Then this number will be a categorical version of the Hausdorff distance between the images of $\cat{P}$ and $\cat{Q}$ in $\cat{I}$ (Definition~\ref{def:weight-embedding}). 
Finally, the interleaving distance between $F$ and $G$ (Definition~\ref{def:interleaving-distance}) will use a categorical version of Gromov-Hausdorff distance between $\cat{P}$ and $\cat{Q}$ (Definition~\ref{def:gh-distance}).

\subsection{Categorification of metric spaces}

In order to quantify interleavings, we will work with categories in which each of the morphisms has a number, called its \emph{weight}. Furthermore, we will require these weights to be compatible with composition in the category, providing a triangle inequality. We describe such \emph{weighted categories} in detail in Section~\ref{sec:wcat}.
As special cases, we have metric spaces, and the more general, Lawvere metric spaces (see Definition~\ref{defn:lawvere}).
For example, in $\cat{R}$ the weight of $a \leq b$ is $b-a$, and in $\cat{I_{\eps}}$ the weight of $(a,i) \leq_{\eps} (b,j)$ for $i,j \in \{0,1\}$ is $b-a$.

\subsection{Categorification of Hausdorff distance and Gromov-Hausdorff distance} \label{sec:cat-gh-intro}

The standard definition of Hausdorff distance in metric spaces (Definition~\ref{def:hausdorff-metric}) extends verbatim to Lawvere metric spaces (Definition~\ref{def:hausdorff-lawvere}).
For two objects in a weighted category, we can take the infimum of the weights of the morphisms between them to obtain a Lawvere metric space.
We define Hausdorff distance in a weighted category to be the Hausdorff distance in the corresponding Lawvere metric space (Definition~\ref{def:weight-embedding}).
For example, the Hausdorff distance between the two copies of $\cat{R}$ in $\cat{I_{\eps}}$ is $\eps$.

So for pairwise embeddings of weighted categories, $\cat{P} \hookrightarrow \cat{I} \hookleftarrow \cat{Q}$, we have the Hausdorff distance between the images of $\cat{P}$ and $\cat{Q}$. Taking the infimum over all such pairwise embeddings, we arrive at the  Gromov-Hausdorff distance between $\cat{P}$ and $\cat{Q}$ (Definition~\ref{def:gh-distance}).
The interleaving distance between functors $F:\cat{P} \to \cat{C}$ and $G:\cat{Q} \to \cat{C}$ where $\cat{P}$ and $\cat{Q}$ are weighted categories is given by the infimum over pairwise embeddings for which there exists an extension (Definition~\ref{def:interleaving-distance}).
For example, if persistence modules $M$ and $N$ are $\cat{I_{\eps}}$-interleaved, then their interleaving distance is at most $\eps$.

\subsection{Categorification of the space of metric spaces}

To show that interleaving distance satisfies the triangle inequality (Theorem~\ref{thm:interleaving-lawvere}), we define and study a bicategory of pairwise embeddings (Section~\ref{sec:embedding-pair}) and a bicategory of pairwise weighted embeddings which we call the \emph{Gromov-Hausdorff bicategory} (Section~\ref{sec:gromov}). As a consequence of this structural theory, we are able to show  that the Gromov-Hausdorff distance between weighted categories and the interleaving distance between weighted functors with fixed codomain
satisfy the triangle inequality -- in fact, both this class of weighted categories, and this class of functors are Lawvere metric spaces (Theorems \ref{thm:gh-lawvere} and \ref{thm:interleaving-lawvere}).
Furthermore, interleaving and interleaving distance are \emph{stable} (Theorems \ref{thm:I-stability} and \ref{thm:interleaving-stability}).

\subsection{Comparison with previous notions of interleaving}

To compare these notions with the interleavings presented in~\cite{bdss:1}, we consider the future equivalences studied by Marco Grandis~\cite{grandis:2005,grandis:2007} and show that they too can be assembled into a bicategory (Section~\ref{sec:fut}). 
Following our earlier development, this can be extended to weighted categories, giving another distance, the \emph{future equivalence distance}, between weighted categories and between functors indexed by weighted categories with fixed codomains (Section~\ref{sec:wfut}). These distances also satisfy the triangle inequality and the corresponding classes of weighted categories and classes of functors are Lawvere metric spaces (Theorems \ref{thm:fut-lawvere} and \ref{thm:fut-functor-lawvere}). Finally we show that these two approaches are compatible (Section~\ref{sec:fut-to-emb}). That is, from a future equivalence we may construct an equivalent interleaving (Proposition~\ref{prop:fut-emb}). We obtain a full and faithful functor from the category of future equivalences between categories $\cat{P}$ and $\cat{Q}$ to the category of pairwise embeddings of $\cat{P}$ and $\cat{Q}$ (Theorem~\ref{thm:phi-full-faithful}).
Furthermore, these assemble to provide a functor of bicategories from the bicategory of future equivalences to the bicategory of pairwise embeddings (Section~\ref{sec:functor-bicategories}).
Using this theory we see that the future equivalence distance provides an upper bound for the Gromov-Hausdorff distance and the interleaving distance (Theorems \ref{thm:gh-vs-fut} and \ref{thm:interleaving-vs-fut}).

\subsection{Connections with previous work in persistent homology}
\label{sec:conn-with-ph}

In this section we show how our view of interleaving can be used to help understand some constructions in the persistent homology literature.

\subsubsection*{Interleaving and interleaving distance}

Let $M,N: \cat{R} \to \cat{Vect}$ be two persistence modules and let $\eps \geq 0$.
As observed in Section~\ref{sec:interleaving}, $M$ and $N$ are $\eps$-interleaved if and only if they are $\cat{I_{\eps}}$-interleaved.
The interleaving distance of $M$ and $N$, denoted $d(M,N)$, is defined by setting $d(M,N) \leq \eps$ if and only if $M$ and $N$ are $\eps'$-interleaved for all $\eps' > \eps$.
Using our categorical notion of interleaving, we are able to provide a more succinct definition.

Consider the category $\cat{I_{\eps^+}}$ given by the poset $(\R \amalg \R,\leq_{\eps^+})$, where if $a$ and $b$ are in the same copy of $\R$ then $a\leq_{\eps^+} b$ if and only if $a \leq b$ and if $a$ and $b$ are in different copies of $\R$ then $a \leq_{\eps^+} b$ if and only if $\eps < b-a$.

\begin{lem} \label{lemma:interleaving}
  $d(M,N) \leq \eps$ if and only if $M$ and $N$ are $\cat{I_{\eps^+}}$-interleaved. That is, we have the following commutative diagram.
\begin{equation} \label{cd:interleaving-dist}
	\begin{tikzcd}[row sep=tiny]
		& \cat{I_{\eps^+}} \arrow[dd,dashed]	\\
		\cat{R} \arrow[ur,hook,"i_1"] \arrow[dr,"M"'] & & \cat{R} \arrow[ul,hook',"i_2"'] \arrow[dl,"N"] \\
		& \cat{Vect}
	\end{tikzcd}
\end{equation}
\end{lem}

\begin{proof}
  ($\Leftarrow$) Let $\eps' > \eps$. Then $I_{\eps'}$ is a subcategory of $I_{\eps_+}$ and 
% the inclusions of $\ cat{R}$ commute with the inclusion $I_{\eps'} \incl I_{\eps_+}$.
we have the following commutative diagram
\begin{equation*} 
	\begin{tikzcd}[row sep=tiny]
		& \cat{I_{\eps'}} \arrow[dd,hook] \\ \\	
		& \cat{I_{\eps^+}} \arrow[dd]	\\
	\cat{R} \arrow[ur,hook] \ar[uuur,hook] \arrow[dr,"M"'] & & \cat{R} \arrow[ul,hook'] \arrow[uuul,hook'] \arrow[dl,"N"] \\
		& \cat{Vect}
	\end{tikzcd}
\end{equation*}

  ($\Rightarrow$) Let $\eps'' > \eps' > \eps$. Then we have the following commutative diagram.
\begin{equation*} 
	\begin{tikzcd}[row sep=tiny]
		& \cat{I_{\eps''}} \arrow[dd,hook] \\ \\	
		& \cat{I_{\eps'}} \arrow[dd]	\\
	\cat{R} \arrow[ur,hook] \ar[uuur,hook] \arrow[dr,"M"'] & & \cat{R} \arrow[ul,hook'] \arrow[uuul,hook'] \arrow[dl,"N"] \\
		& \cat{Vect} 
	\end{tikzcd} 
\end{equation*}
Observing that $\cat{I_{\eps^+}} = \colim_{\eps'>\eps} \cat{I_{\eps'}}$ we obtain the desired commutative diagram~\eqref{cd:interleaving-dist}. 
\end{proof}

We remark that the Hausdorff distance between the two copies of $\cat{R}$ in $\cat{I_{\eps^+}}$ is $\eps$.

\subsubsection*{Interleaving and weak interleaving}

In~\cite{ccsggo:interleaving}, the authors define a discrete variant of interleaving they call weak $\eps$-interleaving. This notion can be nicely described using our categorical language, which furthermore clarifies its connection to the usual interleaving defined above.

For fixed $\alpha \in \R$, let $\cat{I_{\alpha,\eps}}$ be the category given by the poset $(\R \amalg \R, \leq_{\alpha,\eps})$ where if $a$ and $b$ are in the same copy of $\R$ then $a \leq_{\alpha,\eps} b$ if and only if $a \leq b$. The remaining inequalities are generated by the following: for all $n \in \mathbb{Z}$, $(\alpha+2n\eps,1) \leq_{\alpha,\eps} (\alpha + (2n+1)\eps,0) \leq_{\alpha,\eps} (\alpha+(2n+2)\eps,1)$.
Say that persistence modules $M$, $N$ are \emph{weak $\eps$-interleaved} if they are $\cat{I_{\alpha,\eps}}$-interleaved for some $\alpha$.

\begin{lem}
  If $M$ and $N$ are weak $\eps$-interleaved then they are $3\eps$-interleaved.
\end{lem}

\begin{proof}
  Observe that for all $\alpha \in \R$, $\cat{I_{3\eps}}$ is a subcategory of $\cat{I_{\alpha,\eps}}$ and that we have the following commutative diagram.
\begin{equation*} 
	\begin{tikzcd}[row sep=tiny]
		& \cat{I_{3\eps}} \arrow[dd,hook] \\ \\	
		& \cat{I_{\alpha,\eps}} \arrow[dd]	\\
	\cat{R} \arrow[ur,hook] \ar[uuur,hook] \arrow[dr,"M"'] & & \cat{R} \arrow[ul,hook'] \arrow[uuul,hook'] \arrow[dl,"N"] \\
		& \cat{Vect}
	\end{tikzcd} \qedhere
\end{equation*}
\end{proof}

\subsubsection*{Old and new definitions of interleaving distance}

The standard definition of interleaving definition, as at the beginning of Section~\ref{sec:interleaving} and in Lemma~\ref{lemma:interleaving}, does not exactly match our more general definition outlined at the end of Section~\ref{sec:interleaving} and the end of Section~\ref{sec:cat-gh-intro}.

For example, consider the persistence modules $M_c$ over the field $K$ where $c \in \R$, that are given by $M(a) = 0$ if $a<c$ and $M(a) = K$ if $a \geq c$ and linear maps equal to the identity map whenever possible. 
Then $M_0$ and $M_1$ are $\cat{I_1}$-interleaved, but are not $\cat{I_{\eps}}$-interleaved for $\eps < 1$. So the (old) interleaving distance is between $M_0$ and $M_1$ is $1$.

However, using our more general notion, we obtain an interleaving distance of $0$. Indeed, we have the following commutative diagram.
\begin{equation*} 
	\begin{tikzcd}[row sep=tiny]
		& \cat{I_0} \arrow[dd,dashed]	\\
		\cat{R} \arrow[ur,hook,"j_1"] \arrow[dr,"M_0"'] & & \cat{R} \arrow[ul,hook',"j_2"'] \arrow[dl,"M_1"] \\
		& \cat{Vect}
	\end{tikzcd}
\end{equation*}
Here $j_0(a) = (a,0)$ as before, but $j_1(a) = (a-1,1)$.
Then the Hausdorff distance between the two copies of $\cat{R}$ in $\cat{I_0}$ is $0$ and the interleaving distance between $M_0$ and $M_1$ is $0$.

Notice that the new definition is translation invariant. This may be appropriate in applications where there is no canonical choice of $0$ for the measurements.

\subsubsection*{Observable morphisms}

In~\cite{ccbds}, the authors define observable morphisms of persistence modules. Here we observe that these may also be described using our framework. 

Let $\cat{O}$ denote the category given by the poset $(R \amalg R,\leq_o)$, where if $a$ and $b$ are in the same copy of $\R$ then $a \leq_o b$ if and only if $a \leq b$, $(a,0) \leq_o (b,1)$ if and only if $a < b$ and we never have that $(a,1) \leq_o (b,0)$.

Then an observable morphism from a persistence module $M$ to a persistence module $N$ is the following commutative diagram.
\begin{equation*} 
	\begin{tikzcd}[row sep=tiny]
		& \cat{O} \arrow[dd,dashed]	\\
		\cat{R} \arrow[ur,hook,"i_1"] \arrow[dr,"M"'] & & \cat{R} \arrow[ul,hook',"i_2"'] \arrow[dl,"N"] \\
		& \cat{Vect}
	\end{tikzcd}
\end{equation*}

\subsection{Applications}

As an application, we show that a notion of shift equivalences of discrete dynamical systems~\cite{Williams:1973,Wagoner:1999,bush:thesis} fits nicely in our theory (Section~\ref{sec:dynamical}).

For another potential application consider
% Riemannian manifolds $X$ and $Y$ with posets of open sets $\cat{\mathcal{O}(X)}$ and $\cat{\mathcal{O}(Y)}$. The Riemannian structure on $X$ and $Y$ induces a Lawvere metric on $\cat{\mathcal{O}(X)}$ and $\cat{\mathcal{O}(Y)}$~\cite{bdss:1}.  If $F$ and $G$ are
sheaves, presheaves, cosheaves, or precosheaves, $F$ and $G$, on
different metric spaces, $X$ and $Y$, but with common codomain $\cat{C}$.
Then our theory provides an interleaving distance $d(F,G)$.
% The special case of this where $X = Y  = \mathbb{R}$ is of interest to the study of Reeb graphs and merge trees, in which some recent progress~\cite{deSilva:2016,Munch:2015} has used the authors' previous work in this direction~\cite{bdss:1}.

\subsubsection*{Future work}

One of the most important uses of Gromov-Hausdorff distance is to study the convergence of sequences of metric spaces. We would like to suggest that the theory presented here may allow one to study the convergence of weighted categories, and the convergence of functors on weighted categories with fixed codomain.
 
\subsubsection*{Outline of the paper}

In Section~\ref{sec:background}, we define terms and constructions that we will use throughout.
In Section~\ref{sec:hausdorff}, we give a progression of Hausdorff distances, on metric spaces, Lawvere metric spaces, and weighted categories.
In Section~\ref{sec:embedding-pair}, we consider pairwise embeddings of categories and interleavings of functors.
In Section~\ref{sec:gromov}, we study a categorical version of the space of metric spaces.
In Sections \ref{sec:fut}, \ref{sec:fut-to-emb} and \ref{sec:wfut} we study future equivalences, their connections to embeddings, and weighted future equivalences.
In Section~\ref{sec:applications} we consider some applications.
Finally in Appendix~\ref{sec:enriched-theory} we view some of our constructions from the vantage of enriched category theory, and in Appendix~\ref{sec:hausdorff-interleaving} we show that Hausdorff distance is an example of interleaving distance.

\section{Background}
\label{sec:background}

In this section we discuss 
\emph{Lawvere metric spaces}, a categorical version of metric spaces, 
\emph{weighted categories}, a categorical version of length spaces, and how to obtain the former from the latter.
We also define embeddings of categories, and recall the notion
\emph{bicategory}, a second-order version of category, which turns out to be the proper context for the categorical version of the space of metric spaces.
Finally, we define cospans and pairwise embeddings of categories and of weighted categories.

\subsection{Lawvere metric spaces}
\label{sec:lawvere}

Lawvere~\cite{lawvere:1973} defined the following generalization of metric spaces.

\begin{defn}\label{defn:lawvere}
  A \emph{Lawvere metric space} consists of a class of objects $X$ together with a function $d:X\times X \to [0,\infty]$ such that
  \begin{enumerate}
  \item for all $x \in X$, $d(x,x)=0$, and
  \item for all $x,y,z \in X$, $d(x,z) \leq d(x,y)+d(y,z)$.
  \end{enumerate}
\end{defn}
Note that we have relaxed the usual definition of a metric in four ways: $d$ is not required to be finite, symmetric or to separate points, and $X$ is not required to be a set. 
%When $X$ is a set, this is also called an extended pseudoquasimetric.
A Lawvere metric space $(X,d)$ is said to be \emph{symmetric} if for all $x,y \in X$, $d(x,y) = d(y,x)$.
We say that a Lawvere metric space $(X,d)$ is \emph{small} if $X$ is a set.
The last two relaxations are less essential: we can restrict to all the points within a finite distance of some fixed point to obtain the finiteness condition, and we can identify points that are distance 0 apart to obtain the separation condition.

A \emph{morphism of Lawvere metric spaces} $F:(X,d) \to (Y,d')$ is a map $F:X \to Y$ satisfying $d'(f(x),f(y)) \leq d(x,y)$ for all $x,y \in X$.
That is, it is a non-expansive, or 1-Lipschitz map.
We denote by $\cat{Lawv}$ the category of (small) Lawvere spaces.

\begin{rmk}
  Lawvere metric spaces may equivalently be defined to be categories enriched in the strict monoidal poset $(([0,\infty],\geq),+,0)$. Details are given in the appendix.
\end{rmk}

\subsection{Weighted categories}
\label{sec:wcat}

Originally called \emph{normed categories} by Lawvere~\cite{lawvere:1973}, our weighted categories are what Grandis calls \emph{additive} in ~\cite[Section 3]{grandis:2007}.

\begin{defn}\label{defn:weighted-category}
A \emph{weighted category} is a category $\cat{C}$ 
in which every morphism $f$ has a \emph{weight} $w(f) \in [0,\infty]$,
such that
\begin{enumerate}
	\item $w(1_{a})=0$ for any identity morphism $1_{a}$
	\item $w(fg) \leq w(f) + w(g)$.
\end{enumerate} 
\end{defn}

A morphism of weighted categories is a functor $F: \cat{C} \rightarrow \cat{D}$ with 
\[
	w_{\cat{D}}(F(c\rightarrow c')) \leq w_{\cat{C}}(c \rightarrow c').
\]
We say that such a functor $F$ is \emph{nonexpansive}.
We denote the category of weighted small categories and nonexpansive morphisms by $\cat{wCat}$.

\begin{ex}
  Let $(X,d)$ be a metric space. Let $\cat{C}$ be the category whose objects are the points in $X$ and such that each morphism set $\cat{C}(x,y)$ contains a single element whose weight is given by $d(x,y)$. Then $\cat{C}$ is a weighted category. This construction generalizes to Lawvere metric spaces (Definition~\ref{defn:lawvere}).
\end{ex}

\begin{ex}
  Let $X$ be a topological space with a class of admissible paths $A$ and a length structure $L:A \to [0,\infty]$~\cite{bbi:book}. Let $\cat{C}$ be the category whose objects are points in $X$ and whose morphisms are given by paths in $A$ with weight function given by the length structure. Then $\cat{C}$ is a weighted category.
\end{ex}

\begin{rmk}\label{rmk:wset}
A weighted category can be considered to be a category enriched in ``weighted sets''. Details are given in the appendix.  
\end{rmk}

\subsection{From weighted categories to Lawvere metric spaces}
\label{sec:wcat-to-lawvere}

Let $(\cat{C},w)$ be a weighted category. Let $C$ denote the class of objects in $\cat{C}$. The \emph{induced metric} given by
\begin{equation*}
  d(x,y) = \inf_{f:x\to y} w(f).
\end{equation*}
makes $(C,d)$ into a Lawvere metric space.

Let $F:(\cat{P},w) \to (\cat{Q},w')$ be a morphism of weighted categories, and let $(P,d)$ and $(Q,d')$ be the corresponding Lawvere metric spaces. 
Then for all $a,b \in P$,
\begin{equation*}
  d'(F(a),F(b)) = \inf_{g:F(a)\to F(b)}w'(g) \leq \inf_{f:a\to b}w'(F(f)) \leq \inf_{f:a \to b}w(f) = d(a,b).
\end{equation*}
Thus $F$ induces a morphism of Lawvere metric spaces $F:(P,d) \to (Q,d')$.

\begin{rmk}
The above construction, $(\cat{C}, w) \rightarrow (C,d)$, defines a functor, $\cat{wCat} \rightarrow \cat{Lawv}$.
The verifications are straightforward.
For an enriched-category theoretic interpretation of this functor, see  Appendix~\ref{sec:comparison}.
\end{rmk}

\subsection{Embeddings}
\label{sec:embedding}

Let $F : \cat{A} \rightarrow \cat{C}$ be a functor.  
We will say that $F$ is an \emph{embedding} if it is full, faithful, and injective on objects.
We denote embeddings by hooked arrows,  $F : \cat{A} \hookrightarrow \cat{C}$.
An \emph{embedding of weighted categories} is a weight-preserving functor that is an embedding of the underlying categories.
An \emph{embedding of Lawvere metric spaces} is an isometric injective map.

\begin{lem} \label{lem:embedding}
  Let $F: (\cat{P},w) \hookrightarrow (\cat{Q},w')$ be an embedding of weighted
  categories.  Then the induced map of Lawvere metric spaces
  $F: (P,d) \to (Q,d')$ is also an embedding.
\end{lem}

\begin{proof}
  Since the functor $F:\cat{P} \to \cat{Q}$ is injective on objects, $F:(P,d) \to (Q,d')$ is injective.
  Since $F$ is full and weight-preserving, we have for all $a,b \in P$,
\begin{equation*}
  d'(F(a),F(b)) = \inf_{g:F(a)\to F(b)}w'(g) = \inf_{f:a\to b}w'(F(f)) = \inf_{f:a \to b}w(f) = d(a,b).
\end{equation*}
So the induced map is an isometry.
\end{proof}

\subsection{Bicategories}
\label{sec:bicat}

A \emph{strict 2-category} is a category in which every hom-set is a small category. More precisely, it is a category enriched in $\cat{Cat}$, the category of small categories. 
A \emph{bicategory} is also a category in which every hom-set is a small category, but in which the associativity and unity laws of enriched categories only hold up to coherent isomorphism. 
For details, see~\cite{benabou:1967}.

\subsection{Cospans}
\label{sec:cospan}

Let $P$ and $Q$ be objects in the category $\cat{C}$.
The category $\cat{Cospan}(P,Q)$ consists of all diagrams $P \rightarrow X \leftarrow Q$, with commuting diagrams
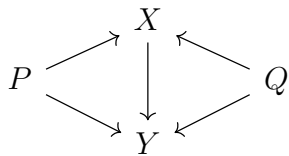
\[
	\begin{tikzcd}[row sep=tiny]
		& X \arrow[dd]	\\
		P \arrow[ur] \arrow[dr] & & Q \arrow[ul] \arrow[dl] \\
		& Y
	\end{tikzcd}
\]
as morphisms.
If $\cat{C}$ has push-outs, we may form the bicategory $\cat{Cospan}$, with horizontal composition defined by the following diagram.
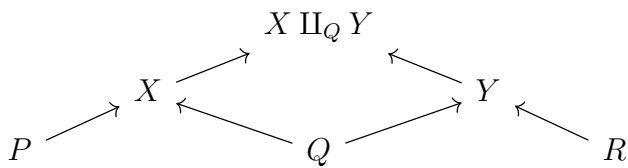
\[
	\begin{tikzcd}[row sep=tiny]
		& & X \amalg_{Q} Y \\
		& X \arrow[ur] & & Y \arrow[ul] \\
		P \arrow[ur] & & Q \arrow[ur] \arrow[ul] & & R \arrow[ul]
	\end{tikzcd}
\]

\subsection{Pairwise embeddings of categories and weighted categories}
\label{sec:Emb-wEmb}

For small categories $P$ and $Q$, define
the \emph{category of pairwise embeddings of $P$ and $Q$}, 
$\cat{Emb}(P,Q)$, to be
the full subcategory of $\cat{Cospan}(P,Q)$ (Section~\ref{sec:cospan}) consisting of diagrams $P \xrightarrow{F} I \xleftarrow{G} Q$ in which $F$ and $G$ are embeddings (Section~\ref{sec:embedding}). 

% \begin{rmk}
% Strictly speaking, we should refer to the above cospan by $(F,G)$, to avoid ambiguity.
% \end{rmk}

If $P$ and $Q$ are weighted categories (Section~\ref{sec:wcat}), let $\cat{wEmb}(P,Q)$ be the category of all \emph{weighted pairwise embeddings of $P$ and $Q$}, in which the objects are diagrams $P \hookrightarrow I \hookleftarrow Q$ of weighted embeddings (Section~\ref{sec:embedding}).
The morphisms are as in $\cat{Emb}(P,Q)$, only weighted.

\section{Categorification of Hausdorff distance}
\label{sec:hausdorff}

In this section we recall the definition of the \emph{Hausdorff distance} and discuss its generalization to Lawvere metric spaces.  
We show that the collection of subsets of a Lawvere metric space, equipped with the Hausdorff distance, is itself a Lawvere metric space.
Finally, we define Hausdorff distance for weighted categories by reducing to Lawvere metric spaces.

\subsection{Hausdorff distance in metric spaces}
\label{sec:hausdorff-metric}

In this section we give two equivalent definition of Hausdorff distance for subsets of a metric space. In Appendix~\ref{sec:hausdorff-interleaving} we give a third equivalent definition using the notion of interleaving distance from~\cite{bdss:1}.

Let $(M,d)$ be a metric space. Let $A$ and $B$ be subsets of $M$. 
\begin{defn} \label{def:hausdorff-metric}
  The \emph{Hausdorff distance} between $A$ and $B$ is given by 
  \begin{equation*}
    d_H(A,B) = \max \left( \sup_{a\in A} \inf_{b\in B} d(a,b), \sup_{b\in B} \inf_{a\in A} d(a,b) \right).
  \end{equation*}
\end{defn}

\begin{defn}
  For a subset $A \subset M$ and $r \geq 0$, let the \emph{$r$-offset} of $A$ be given by
  \begin{equation*}
    A^r = \{m \in M \ | \ \inf_{a\in A} d(a,m) \leq r\}.
  \end{equation*}
\end{defn}
% Then it is a standard exercise to show establish 
We have the following equivalent definition of Hausdorff distance.

\begin{lem} \label{lem:hausdorff-metric} 
  Let $A,B \subset M$. Then
  \begin{equation*}
    d_H(A,B) = \inf \{ r \geq 0 \ | \ A \subset B^r \text{ and } B \subset A^r \}.
  \end{equation*}
\end{lem}

\begin{proof}
  ($\leq$) 
Assume $A \subset B^r$ and $B \subset A^r$. Let $a \in A$. For $\eps > 0$ there exists $b \in B$ such that $d(a,b) \leq r + \eps$. 
So $\inf_{b\in B} d(a,b) \leq r$.
Similarly, for all $b \in B$, $\inf_{a\in A} d(a,b) \leq r$.
Thus $d_H(A,B) \leq r$.

  ($\geq$) 
Assume $A \not\subset B^r$. Then there exists $a \in A$ and $\eps > 0$ such that $\inf_{b\in B} d(a,b) \geq r + \eps$. Thus $d_H(A,B) \geq r+\eps$.
Similarly, this is implied by $B \not\subset A^r$.
\end{proof}

\subsection{Hausdorff distance in Lawvere metric spaces}
\label{sec:hausdorff-lawvere}

In this section, we generalize Hausdorff distance from metric spaces to Lawvere metric spaces;  the resulting distance is not necessarily symmetric.

Let $(M,d)$ be a Lawvere metric space. Let $A$ and $B$ be subsets of $M$. 
\begin{defn} \label{def:hausdorff-lawvere}
  The \emph{Hausdorff distance} from $A$ to $B$ is given by 
  \begin{equation*}
    d_H(A,B) = \max \left( \sup_{a\in A} \inf_{b\in B} d(a,b), \sup_{b\in B} \inf_{a\in A} d(a,b) \right).
  \end{equation*}
\end{defn}

Note that this is exactly the same as the definition for metric spaces. To state the second version of this definition requires more care since there are non-symmetrical offsets.

\begin{defn}
  For a subset $A \subset M$ and $r \geq 0$, let the \emph{future $r$-offset} of $A$ be given by
  \begin{equation*}
    A^r = \{m \in M \ | \ \inf_{a\in A} d(a,m) \leq r\}.
  \end{equation*}
The \emph{past $r$-offset} of $A$ is given by
  \begin{equation*}
    {}^r\!A = \{m \in M \ | \ \inf_{a\in A} d(m,a) \leq r\}.
  \end{equation*}
\end{defn}

We will need the following observations.

\begin{lem} \label{lem:offset}
  \begin{enumerate}
  \item $(A^r)^s \subset A^{r+s}$
    and ${}^s(^r\!A) \subset {}^{r+s}\!A$
  \item If $A \subset B$ then $A^r \subset B^r$ and ${}^r\!A \subset {}^r\!B$.
  \end{enumerate}
\end{lem}

\begin{proof}
  \begin{enumerate}
  \item Let $x \in (A^r)^s$. For $\eps > 0$, there exists $y \in A^r$ such that $d(y,x) \leq s+\eps/2$ and there exists $a \in A$ such that $d(a,y) \leq r+\eps/2$. By the triangle inequality, $d(a,x) \leq r+s+\eps$. Thus $x \in A^{r+s}$. The second containment follows similarly.
  \item Assume $A \subset B$. Let $x \in A^r$. For $\eps>0$, there exists $a \in A \subset B$ such that $d(a,x) \leq r+\eps$. Thus $x\in B^r$. The second containment follows similarly. \qedhere
  \end{enumerate}
\end{proof}

We now have the following equivalent definition of Hausdorff distance.

\begin{lem} \label{lem:hausdorff-lawvere}
  Let $A,B \subset M$. Then
  \begin{equation*}
    d_H(A,B) = \inf \{ r \geq 0 \ | \ A \subset {}^r\!B \text{ and } B \subset A^r \}.
  \end{equation*}
\end{lem}

\begin{proof}
  The proof is the exactly the same as the proof of Lemma~\ref{lem:hausdorff-metric} after replacing $B^r$ with ${}^r\!B$.
\end{proof}

\begin{thm} \label{thm:hausdorff}
  The class of subspaces of a Lawvere metric space together with the above Hausdorff distance is a Lawvere metric space.
\end{thm}

\begin{proof}
  Let $A \subset M$. Then $d_H(A,A)=0$. 

  Let $A,B,C \subset M$. Let $r>d_H(A,B)$ and $s>d_H(B,C)$. Then $A \subset {}^r\!B$, $B\subset A^r$ and $C \subset {}^sC$, $C \subset B^s$. Thus, by Lemma~\ref{lem:offset}, $A \subset {}^r\!B \subset {}^r(^sC) \subset {}^{r+s}C$ and $C\subset B^s \subset (A^r)^s \subset A^{r+s}$. Therefore $d_H(A,C) \leq r+s$. 
\end{proof}

Symmetrizing this distance we obtain the following.

\begin{defn}
  The \emph{symmetric Hausdorff distance} is given by
  \begin{equation*}
    \tilde{d}_H(A,B) = \max(d_H(A,B),d_H(B,A)).
  \end{equation*}
\end{defn}

\begin{thm}
  The class of subspaces of a Lawvere metric space together with the symmetric Hausdorff distance is a symmetric Lawvere metric space.
\end{thm}

\begin{proof}
  For $A \subset M$, $\tilde{d}_H(A,A) = 0$. For $A,B,C \subset M$ we have $\tilde{d}_H(A,C) \leq \max( d_H(A,B) + d_H(B,C), d_H(C,B) + d_H(B,A)) \leq \max(d_H(A,B),d_H(B,A)) + \max(d_H(B,C),d_H(C,B)) = \tilde{d}_H(A,B) + \tilde{d}_H(B,C)$. By definition, $\tilde{d}_H$ is symmetric.
\end{proof}

\begin{rmk}
  $d_H(A,B) \leq \tilde{d}_H(A,B)$.
\end{rmk}

% We will later need the following.

% \begin{lem} \label{lem:hausdorff-subset}
%   Let $I \subset J$ be Lawvere metric spaces.
%   Let $d_H^I$ and $d_H^J$ denote the corresponding Hausdorff distances.
%   Then for $P,Q \subset I$, $d_H^J(P,Q) = d_H^I(P,Q)$.
% \end{lem}

% \[
% 	d_{I}(p,Q) = \inf_{q \in Q} d_{I}(p,q)
% \]
% \[
% 	d_{I}^{\ell}(P,Q) = \sup_{p \in P} d_{I}(p,Q).
% \]
% and
% \[
% 	d_{I}(P,q) = \inf_{p \in P} d_{I}(p,q)
% \]
% then
% \[
% 	d_{I}^{r}(P,Q) = \sup_{q \in Q} d_{I}(P,q).
% \]
% Finally,
% \[
% 	d_H^I(P,Q) = \max\{ d_{I}^{\ell}(P,Q), d_{I}^{r}(P,Q) \}.
% \]

\subsection{Hausdorff distance in weighted categories}
\label{sec:hausdorff-wcat}

We continue the progression, generalizing Hausdorff distance from Lawvere metric spaces to weighted categories.

\begin{defn} \label{def:weight-embedding}
Let $P \hookrightarrow I \hookleftarrow Q$ %\in \cat{wEmb}(P,Q)$
be a pairwise embedding of weighted categories $P$ and $Q$ (Section~\ref{sec:Emb-wEmb}).
By Lemma~\ref{lem:embedding}, there is a corresponding pairwise embedding of Lawvere metric spaces.
% (see Sections \ref{sec:wcat-to-lawvere} and \ref{sec:embedding}).
Let $d_H^I(P,Q)$ and $\tilde{d}_H^I(P,Q)$
denote the Hausdorff distance and symmetric Hausdorff distance, respectively,
in $I$ between the isometric images of $P$ and $Q$ 
(Section~\ref{sec:hausdorff-lawvere}).
\end{defn}

% The weight for this ``arrow'' will be the (directed) Hausdorff distance from $P$ to $Q$ in $I$.  
% Specifically, for $x,y \in I$, let
% \[
% 	d_{I}(x,y) = \inf_{f:x \rightarrow y} w_{I}(f)
% \]
% (this makes $I$ into a Lawvere metric space in the usual way).
% Note that $d_{I}$ need to be symmetric.

\section{Embedding pairs and interleaving of functors} \label{sec:embedding-pair}

In this section we define categories of pairwise embeddings and a bicategory of embeddings. We use this to define interleavings of functors and prove statements which will imply the triangle inequality and stability for interleaving distance once we switch to the weighted setting in Section~\ref{sec:gromov}.

\subsection{The bicategory of embedding pairs}

Recall that for small categories $P$ and $Q$, $\cat{Emb}(P,Q)$ is
the full subcategory of $\cat{Cospan}(P,Q)$ whose arrows are embeddings (Section~\ref{sec:Emb-wEmb}).
We would like to define a bicategory, $\cat{Emb}$,  with all small categories as $0$-cells, and hom categories $\cat{Emb}(P,Q)$, as a sub-bicategory of $\cat{Cospan}$.
In $\cat{Cospan}$, horizontal composition is achieved via the push-out in $\cat{Cat}$.
Explicitly, the composition
\[
	\cat{Cospan}(Q,R) \times \cat{Cospan}(P,Q) \xrightarrow{\circ} \cat{Cospan}(P,R)
\]
is represented in the diagram
\[
	\begin{tikzcd}[row sep=tiny]
		& & J \circ I \\
		& I \arrow[ur,"i"] & & J \arrow[ul,"j"'] \\
		P \arrow[ur, hook, "F"] & & Q \arrow[ul, hook',"G"'] \arrow[ur, hook, "H"] & & R. \arrow[ul, hook', "K"']
	\end{tikzcd}
\]
Since $G$ and $H$ are injective on objects, we have
\[
	(J \circ I)_{0} \cong (J \setminus Q) \amalg Q \amalg (I \setminus Q).
\]
By~\cite{trnkova:1965}, $i$ and $j$ are embeddings.
In particular,
\[
	(J \circ I)(x,y) \cong
		\left\{
			\begin{array}{ll}
				I(x,y) & \text{if }x,y\in I \\
				J(x,y) & \text{if }x,y \in J
			\end{array}
		\right.
\]
with no ambiguity if $x,y \in Q$.

\begin{prop}\label{prop:pushout-reln}
Suppose $x \in I \setminus Q$ and $y \in J \setminus Q$. 
Then 
\[
	(J \circ I) ( x, y) = \prod_{q \in Q} J(q,y) \times I(x,q) /\sim
\]
where $\sim$ is the smallest equivalence relation in which $(q \xrightarrow{g} y, x \xrightarrow{f} q) \sim (q' \xrightarrow{g'} y, x \xrightarrow{f'} q')$ if there exists some $h : q \rightarrow q'$ such that the diagram
\[
	\begin{tikzcd}
		x \arrow[r,"f"] \arrow[dr,"f'"']
			& q   \arrow[d,"h"] \arrow[dr,"g"]	\\
			& q' \arrow[r,"g'"'] & y
	\end{tikzcd}
\]
commutes.
\end{prop}

\begin{rmk}
Another way to express the equivalence relation in Proposition~\ref{prop:pushout-reln} is to say that $(gh,f) \sim (g,hf)$ for all $f : I \rightarrow Q$, $g : q \rightarrow y$, and $h : q \rightarrow q'$.
\end{rmk}

\begin{notn}
We denote the equivalence class of $(g,f)$ by $[g,f]$.
\end{notn}

\begin{lem}\label{lem:comp-well-def}
With the notation above, if $g : q \rightarrow y$ and $f : x \rightarrow q$, then $g \circ f = [g,f]$ is well defined in $J \circ I$.
\end{lem}

\begin{proof}
To be careful, we use the embeddings $i : I \rightarrow J \circ I \leftarrow J: j$. 
Suppose $(f,g)\sim(f',g')$.  
Then there is some $h : q \rightarrow q'$ such that $g = g'H(h)$ in $J$ and $f' = G(h) f$ in $I$.
Therefore
\begin{eqnarray*}
	j(g)i(f) & = & j( g'H(h))i(f)	\\
		& = & j(g')jH(h)i(f)	\\
		& = & j(g')iG(h)i(f)	\\
		& = & j(g')i(G(h)f)	\\
		& = & j(g')i(f'). 
\end{eqnarray*}
\end{proof}

Since composites of embeddings are embeddings, it follows that the horizontal composition in $\cat{Cospan}$ in fact lies in $\cat{Emb}$.
Similarly, $\cat{Emb}$ is closed under the unitor and associator morphisms in $\cat{Cospan}$, and so forms a sub-bicategory.

For future reference, we note that the unitor $\cat{Emb}_{0} \rightarrow \cat{Emb}_{1}$ is given by $P \mapsto ( P \xrightarrow{1_{P}} P \xleftarrow{1_{P}} P)$.

\subsection{Interleavings of functors}
\label{sec:I-interleaving}

Let $\cat{C}$ be a category.
Let $F : P \rightarrow \cat{C}$, $G : Q \rightarrow \cat{C}$ be functors with common codomain.
Let $(P \hookrightarrow I \hookleftarrow Q) \in \cat{Emb}(P,Q)$.
Abusing notation, we say that $F$ and $G$ are $I$-\emph{interleaved} if there exists a functor $H : I \rightarrow \cat{C}$ such that the diagram
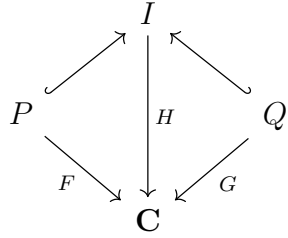
\[
	\begin{tikzcd}
		& I \arrow[dd, "H"]  \\
		P \arrow[dr, "F"'] \arrow[ur, hook]
		& &
		Q \arrow[ul, hook'] \arrow[dl, "G"] \\
		& \cat{C}
	\end{tikzcd}
\]
commutes.

The following result follows from the push-out property and will eventually yield the triangle inequality.

\begin{prop}
	Let $P \hookrightarrow I \hookleftarrow Q$ and $Q \hookrightarrow J \hookleftarrow R$ be embeddings.
	Let $F: P \rightarrow \cat{C}$, $G: Q \rightarrow \cat{C}$, $H : R \rightarrow \cat{C}$ be functors with common codomain.
	If $F$ and $G$ are $I$-interleaved and $G$ and $H$ are $J$-interleaved, then $F$ and $H$ are $J \circ I$-interleaved.
\end{prop}

\begin{thm}[Stability] \label{thm:I-stability}
	If $F$ and $G$ are $I$-interleaved functors and $H : \cat{C} \rightarrow \cat{A}$ is a functor, then $HF$ and $HG$ are $I$-interleaved.
\end{thm}

\section{Categorification of the space of metric spaces} \label{sec:gromov}

We would like to quantify the $I$-interleavings introduced in Section~\ref{sec:I-interleaving}.
To that end, we introduce a weighted version of the bicategory $\cat{Emb}$.
We follow the same progression as in the previous section, layering in the added input of the weights.
That is, we define categories of weighted embedding pairs, and a bicategory of weighted embeddings. 
Furthermore, we use these to define a Gromov-Hausdorff distance of weighted categories, and an interleaving distance between functors, which we prove is stable.

\subsection{The category of weighted embedding pairs}

Recall that for weighted categories $P$ and $Q$ we have a category $\cat{wEmb}(P,Q)$ of weighted pairwise embeddings of $P$ and $Q$ (Section~\ref{sec:Emb-wEmb}).
The next proposition says, in particular, that the forgetful functor $\cat{wCat} \rightarrow \cat{Cat}$ creates push-outs.

\begin{prop}\label{prop:IoJ-in-wCat}
There is a natural weighting on $J \circ I$ that makes the natural functors $I, J \rightarrow J \circ I$ into embeddings.
\end{prop}

\begin{proof}
If $x,y \in I$, then $(J \circ I) (x,y) = I(x,y)$, so we let $w_{J \circ I} = w_{I} : (J \circ I) (x,y) \rightarrow [0,\infty]$.
Similarly, if $x,y \in J$, then $w_{J \circ I} = w_{J} : (J \circ I) (x,y) \rightarrow [0,\infty]$.
This is well defined in $Q$ since the functors $Q \hookrightarrow I,J$ are weight preserving.
Suppose $x \in I \setminus Q$ and $y \in J \setminus Q$.
Let $[g,f] \in (J \circ I)(x,y)$, where $f : x \rightarrow q$ and $g : q \rightarrow y$ for some $q \in Q$.
We set
\[
	w_{J \circ I} ([g,f]) = \inf_{(g',f') \sim (g,f)}
		w_{J}(g') + w_{I}(f').
\]
We define $w_{J \circ I} : (I\circ J) (y,x) \rightarrow [0,\infty]$ in a symmetric fashion.

Clearly, the weight of all identity arrows is zero, since they inherit their weights from either $I$ or $J$.

To see that $w_{J \circ I}$ satisfies sub-additivity, we need to check three cases; there are another three cases that are symmetric, and for composition of morphisms both in $I$ or both in $J$, sub-additivity follows by inheritance. 

First, suppose that $[g,f] : x \rightarrow y$, $h : y \rightarrow z$, where $x \in I$ and $y,z \in J$.
Clearly $h[g,f] = [hg,f]$.  
Note that if $(g'',f'') \sim (g,f)$, then $(hg'',f'') \sim (hg,f)$.  
It follows that 
\begin{eqnarray*}
	w_{J \circ I} ( h[g,f] )
		& = & w_{J \circ I} ( [hg,f] ) \\
		& = & \inf_{(g',f') \sim (hg,f)}
			w_{J}(g') + w_{I}(f') \\
		& \leq & \inf_{(g'',f'') \sim (g,f)}
			w_{J}(hg'') + w_{I}(f'') \\
		& \leq & \inf_{(g'',f'') \sim (g,f)}
			w_{J}(h) + w_{J}(g'') + w_{I}(f'') \\
		& = & w_{J}(h) + \inf_{(g'',f'') \sim (g,f)}
			w_{J}(g'') + w_{I}(f'') \\
		& = & w_{J \circ I}(h) + w_{J \circ I}([g,f]).
\end{eqnarray*}
If $k : w \rightarrow x$, then the argument that $w_{J \circ I} ([g,f]k) \leq w_{J \circ I}([g,f]) + w_{J \circ I}(k)$ is follows the same lines.

Suppose $[g,f]: x \rightarrow y$ and $[k,h]: y \rightarrow z$ for $x,z \in I$, $y \in J$.
Then $[k,h]\circ[g,f] = k(hg)f$ (composition in $I$) for any representatives $(q' \xrightarrow{k} z, y \xrightarrow{h} q') \in [k,h]$ and $(g,f) \in [g,f]$.
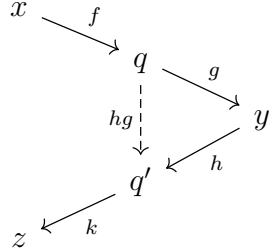
\[
	\begin{tikzcd}[row sep=tiny]
		x \arrow[dr,"f"] \\
		& q \arrow[dr,"g"] \arrow[dd, dashed, "hg"'] \\
		& & y \arrow[dl,"h"] \\
		& q' \arrow[dl,"k"] \\
		z
	\end{tikzcd}
\]
We are grouping $hg$ to indicate that the composite is a morphism in $I$.
Since $Q \hookrightarrow I \hookrightarrow J \circ I$ and $Q \hookrightarrow J$, 
% \begin{eqnarray*}
% 	w_{J \circ I} (hg) & =&  w_{I} (hg)	\\
% 						& = & w_{Q}(hg)	\\
% 						& = & w_{J}(hg) \\
% 						& \leq & w_{J}(h) + w_{J}(g).
% \end{eqnarray*}
\begin{equation*}
	w_{J \circ I} (hg)  =  w_{I} (hg)	
						 =  w_{Q}(hg)
						 =  w_{J}(hg)
						 \leq  w_{J}(h) + w_{J}(g).
\end{equation*}
Therefore
\[
	w_{J \circ I}([k,h]\circ[g,f])	 \leq w_{I}(k) + w_{J}(h) + w_{J}(g) + w_{I}(f)
\]
for any choice of representatives over the equivalence classes.
It follows that 
\begin{eqnarray*}
	w_{J \circ I}([k,h]\circ[g,f])
		& \leq & \inf w_{I}(k) + w_{J}(h) + w_{J}(g) + w_{I}(f)	\\
		& = & w_{J \circ I}([k,h]) + w_{J \circ I}([g,f])
\end{eqnarray*}
where the infimum is taken over all $(k',h') \sim (k,h)$ and $(f',g')\sim(f,g)$.
\end{proof}

\begin{prop}
Equipped with the weighting of Proposition~\ref{prop:IoJ-in-wCat}, $J \circ I$ is the push-out of $I \hookleftarrow Q \hookrightarrow J$ in $\cat{wCat}$.
\end{prop}

\begin{proof}
We have shown that $J \circ I \in \cat{wCat}$.
The underlying unweighted category is the pushout in $\cat{Cat}$.
Consider the following diagram in $\cat{Cat}$,
\[
	\begin{tikzcd}
		Q \arrow[r, hook,"G"] \arrow[d, hook, "H"'] & I \arrow[d, hook, "\lambda"] \arrow[ddr,bend left=30,"\alpha"] \\
		J \arrow[r,hook,"\gamma"] \arrow[drr,bend right=30,"\beta"'] & J \circ I \arrow[dr,dashed,"\exists ! \varphi"] \\
		& & C
	\end{tikzcd}
\]
in which $C$ is a weighted category and $\alpha$ and $\beta$ are nonexpansive.
The unique functor $\varphi:J \circ I \rightarrow C$ of unweighted categories exists by the universal property of the push-out, and it simply remains to verify that $\varphi$ is nonexpansive.

There are essentially three cases to check.
If $f=\lambda(f')$, then $\varphi(f) = \varphi(\lambda(f'))=\alpha (f')$.
Since $\lambda$ is weight-preserving and $\alpha$ is nonexpansive, we have
\[
	w_{C}(\varphi(f)) = w_{C}(\alpha(f')) \leq w_{I}(f') = w_{I\circ J}(f).
\]
Similarly, if $g$ is in the image of $\gamma$, then $w_{C}(\varphi(g)) \leq w_{I\circ J}(g)$.

The remaining case is $[g,f]:x\rightarrow y$ with $x \in I \setminus Q$ and $y \in J \setminus Q$ (or vice versa).
That is, we must show that $w_{C}(\varphi([g,f])) \leq w_{J \circ I}([g,f])$. 
By Lemma~\ref{lem:comp-well-def}, $w_{C}([g,f]) = w_{C}(\beta(g)\alpha(f)) \leq w_{C}(\beta(g)) + w_{C}(\alpha(f)) \leq w_{J}(g) + w_{I}(f)$.
By the universal property of the infimum, $w_{C}([g,f]) \leq w_{J \circ I}([g,f])$, as desired.
\end{proof}

\subsection{The Gromov-Hausdorff bicategory}
\label{sec:bicatGH}

In this section we define the \emph{Gromov-Hausdorff bicategory}, $\cat{wEmb}$.
Let $\cat{wEmb}$ have all small weighted categories for $0$-cells, and hom-categories $\cat{wEmb}(P,Q)$.
Then $\cat{wEmb}$ is a bicategory.
%Then $\cat{wEmb}$ is not a sub-bicategory of $\cat{Cospan}$ since it has more $0$-cells.

Recall that for a pairwise embedding of weighted categories $P \hookrightarrow I \hookleftarrow Q$, we have a Hausdorff distance $d_H^I(P,Q)$ and a symmetric Hausdorff distance $\tilde{d}_H^I(P,Q)$ (Definition~\ref{def:weight-embedding}).

\begin{prop} \label{prop:gh-is-wcat}
The functions $w_{GH}:\cat{wEmb}(P,Q) \rightarrow [0,\infty]$ defined by $w_{GH}(P \hookrightarrow I \hookleftarrow Q) = d_H^I(P,Q)$ determine a weighted (large, locally small) category structure on the $1$-skeleton of $\cat{wEmb}$.
\end{prop}

\begin{proof}

The identity arrow on $P$ is the diagram $P = P = P$
%; clearly the distance from $P$ to $P$ in $P$ is zero.
and $d_H^I(P,P)=0$.

%For horizontal composition, suppose that we have
% \[
% 	\begin{tikzcd}
% 		P \arrow[r,hook,"F"] & I 
% 		& Q \arrow[l,hook',"G"'] \arrow[r,hook,"H"]
% 		& J & R \arrow[l,hook',"K"'].
% 	\end{tikzcd}
% \]
For subadditivity, consider the following commutative diagram
of weighted categories.
\[
	\begin{tikzcd}[row sep=tiny]
		& & J \circ I \\
		& I \arrow[ur,hook,"i"] & & J \arrow[ul,hook',"j"'] \\
		P \arrow[ur, hook, "F"] & & Q \arrow[ul, hook',"G"'] \arrow[ur, hook, "H"] & & R. \arrow[ul, hook', "K"']
	\end{tikzcd}
\]
By Lemma~\ref{lem:embedding} we may also consider this to be the corresponding commutative diagram of Lawvere metric spaces.

We must show that $w_{GH}(J \circ I) \leq w_{GH}(I) + w_{GH}(J)$, or equivalently, that 
\begin{equation} \label{eq:triangle}
  d_H^{J \circ I}(P,R) \leq d_H^{I}(P,Q) + d_H^{J}(Q,R).
\end{equation}
Using the triangle inequality in $J \circ I$ and %Lemma~\ref{lem:hausdorff-subset}, 
Definition~\ref{def:hausdorff-lawvere} 
we have that
$d_H^{J \circ I}(P,R) \leq d_H^{J\circ I}(P,Q) + d_H^{J\circ I}(Q,R) = d_H^I(P,Q) + d_H^J(Q,R).$
\end{proof}

This weighted category has an associated Lawvere metric space (see Section~\ref{sec:wcat-to-lawvere}). Call the corresponding Lawvere metric the \emph{Gromov-Hausdorff distance}. In detail:

\begin{defn} \label{def:gh-distance}
Let $P,Q$ be weighted categories. The \emph{Gromov-Hausdorff distance} between $P$ and $Q$ is given by
\begin{equation*}
  d_{GH}(P,Q) = \inf_{I \in \cat{wEmb}(P,Q)} d_H^I(P,Q).
\end{equation*}
\end{defn}

We remark that as in the usual Gromov-Hausdorff distance, we can restrict to weighted categories $I$ whose underlying category is $P \amalg Q$.

% From Proposition~\ref{prop:gh-is-wcat} we have the following corollary.

\begin{thm} \label{thm:gh-lawvere}
  The class of weighted categories with the Gromov-Hausdorff distance is a Lawvere metric space.
\end{thm}

Now for the symmetric version. 

\begin{prop} \label{prop:sgh-is-wcat}
The functions $\tilde{w}_{GH}:\cat{wEmb}(P,Q) \rightarrow [0,\infty]$ defined by $\tilde{w}_{GH}(P \hookrightarrow I \hookleftarrow Q) = \tilde{d}_H^I(P,Q)$ determine a weighted (large, locally small) category structure on the $1$-skeleton of $\cat{wEmb}$.
\end{prop}

\begin{proof}
  $\tilde{w}(J\circ I) = \tilde{d}_H^{J\circ I}(P,R) = \max(d_H^{J\circ I}(P,R), d_H^{J\circ I}(R,P)) \leq \max( d_H^I(P,Q) + d_H^J(Q,R), d_H^J(R,Q) + d_H^I(Q,P)) \leq \max (d_H^I(P,Q), d_H^I(Q,P)) + \max (d_H^J(Q,R),d_H^J(R,Q) = \tilde{d}_H^I(P,Q) + \tilde{d}_H^J(Q,R) = \tilde{w}_{GH}^I) + \tilde{w}_{GH}^J.$
\end{proof}

Again, this weighted category has an associated Lawvere metric space (see Section~\ref{sec:wcat-to-lawvere}). Call the corresponding Lawvere metric the \emph{symmetric Gromov-Hausdorff distance}. In detail:

\begin{defn}
Let $P,Q$ be weighted categories. The \emph{symmetric Gromov-Hausdorff distance} between $P$ and $Q$ is given by
\begin{equation*}
  \tilde{d}_{GH}(P,Q) = \inf_{I \in \cat{wEmb}(P,Q)} \tilde{d}_H^I(P,Q).
\end{equation*}
\end{defn}

\begin{rmk}
  $d_{GH}(P,Q) \leq \tilde{d}_{GH}(P,Q)$.
\end{rmk}

\begin{thm}
  The class of weighted categories with the symmetric Gromov-Hausdorff distance is a Lawvere metric space.
\end{thm}

\subsection{Interleaving distance of functors from weighted categories}
\label{sec:interleaving-distance-functors}

In this section $\cat{C}, \cat{D}$ are categories and $P,Q$ are weighted categories. Functors $F:P \to \cat{C}$ and $G:Q \to \cat{C}$ are functors from the underlying categories of $P$ and $Q$.

\begin{defn} \label{def:interleaving-distance}
The \emph{interleaving distance} between $F$ and $G$ is given by
\begin{equation*}
  d(F,G) = \inf \{ d_H^I(P,Q) \ | \ I \in \cat{wEmb}(P,Q) \text{ and $F$ and $G$ are $I$-interleaved} \}.
\end{equation*} 
\end{defn}

To compute the interleaving distance, we can restrict to weighted categories $I$ with objects $P_{0} \amalg Q_{0}$.
Indeed, let $i: P \hookrightarrow I$ and $j : Q \hookrightarrow I$ be embeddings.
Let $\hat{I}$ be the category with objects $P_{0} \amalg Q_{0}$ and for all $p \in P$ and $q \in Q$, $\hat{I}(p,q) = I(i(p),j(q))$ as a weighted set.
(We assume that $P$ and $Q$ embed in $\hat{I}$, of course.) 
In particular, if $i(p) = j(q)$, then there is a weight-zero isomorphism from $p$ to $q$.
Obviously, there is a weight-preserving functor $\theta : \hat{I} \rightarrow I$, and $i$ and $j$ each factor through $\theta$. 
Therefore if $F$ and $G$ are $I$-interleaved, they are also $\hat{I}$-interleaved, and $d^{\hat{I}}_{H}(P,Q) = d^{I}_{H}(P,Q)$.

\begin{thm} \label{thm:interleaving-lawvere}
  The class of functors from weighted categories to a fixed category $\cat{C}$ together with the interleaving distance is a Lawvere metric space.
\end{thm}

\begin{proof}
  Let $F:P \to \cat{C}$. Then $d(F,F)=0$ since $F$ and $F$ are $P$-interleaved.
  Let $G:Q \to \cat{C}$ and $H:R \to \cat{C}$.
  If $F$ and $G$ are $I$-interleaved and $G$ and $H$ are $J$-interleaved then $F$ and $H$ are $J\circ I$-interleaved. So together with \eqref{eq:triangle}, we have that $d(F,H) \leq d(F,G)+d(G,H)$.
\end{proof}

\begin{thm}
  Let $F:P \to \cat{C}$ and $G:Q \to \cat{C}$. Then
  \begin{equation*}
    d_{GH}(P,Q) \leq d(F,G).
  \end{equation*}
\end{thm}

\begin{proof}
  This follows directly from the definitions.
\end{proof}

\begin{thm}[Stability] \label{thm:interleaving-stability}
  Let $F:P \to \cat{C}$, $G:Q \to \cat{C}$ and $H:\cat{C} \to \cat{D}$. 
  Then
  \begin{equation*}
    d(HF,HG) \leq d(F,G).
  \end{equation*}
\end{thm}

\begin{proof}
  This follows from Theorem~\ref{thm:I-stability}.
\end{proof}

We also have symmetric versions of these results.

\begin{defn}
The \emph{symmetric interleaving distance} between $F$ and $G$ is given by
\begin{equation*}
  \tilde{d}(F,G) = \inf \{ \tilde{d}_H^I(P,Q) \ | \ I \in \cat{wEmb}(P,Q) \text{ and $F$ and $G$ are $I$-interleaved} \}.
\end{equation*} 
\end{defn}

\begin{rmk}
  $d(F,G) \leq \tilde{d}(F,G)$.
\end{rmk}

\begin{thm}
  The class of functors from weighted categories to a fixed category $\cat{C}$ together with the symmetric  interleaving distance is a Lawvere metric space.
\end{thm}

\begin{thm}
  Let $F:P \to \cat{C}$ and $G:Q \to \cat{C}$. Then
  \begin{equation*}
     \tilde{d}_{GH}(P,Q) \leq \tilde{d}(F,G).
  \end{equation*}
\end{thm}

\begin{thm}[Stability] 
  Let $F:P \to \cat{C}$, $G:Q \to \cat{C}$ and $H:\cat{C} \to \cat{D}$. 
  Then
  \begin{equation*}
    \tilde{d}(HF,HG) \leq \tilde{d}(F,G).
  \end{equation*}
\end{thm}

\section{Future Equivalences}
\label{sec:fut}

In~\cite{bdss:1}, we made use of interleaved preordered sets.  
The interleavings were equipped with a structure that allowed us to quantify them.
It turns out that when one generalizes these ideas from posets to small categories, one recovers Grandis' notion of \emph{future equivalence}~\cite{grandis:2005}, which we now recall.
We will also show that these assemble to give a bicategory, and can be used to define interleavings of functors.

\subsection{The category of future equivalences}
\label{sec:futpq}

Let $P$ and $Q$ be small categories.  
In this section we describe a category, $\cat{Fut}(P,Q)$, consisting of future equivalences between $P$ and $Q$, with appropriate morphisms.
For further details, see~\cite[Section 2]{grandis:2005}.

An object of $\cat{Fut}(P,Q)$ consists of a quadruple, $(\Gamma,\Kappa,\eta,\nu)$, where
\begin{itemize}
	\item $\Gamma : P \rightarrow Q$ and $\Kappa : Q \rightarrow P$ are functors,
	\item $\eta : \Iota_{P} \Rightarrow \Kappa\Gamma$ and $\nu : \Iota_{Q} \Rightarrow \Gamma\Kappa$ are natural transformations, and
	\item we have the coherence conditions, 
	\[
		\Gamma \eta = \nu \Gamma : \Gamma \Rightarrow \Gamma \Kappa \Gamma 
		\quad \text{and} \quad
		\Kappa \nu = \eta \Kappa: \Kappa \Rightarrow \Kappa \Gamma \Kappa.
	\]	
\end{itemize}
These data produce a sort of adjunction, in the sense that we obtain natural transformations of hom functors,
\[
	Q(\Gamma(p),q) \rightarrow P(p,\Kappa(q))
\]
and
\[
	P(\Kappa(q),p) \rightarrow Q(q,\Gamma(p))
\]
given by $f \mapsto K(f) \circ \eta_{p}$ and $g \mapsto \Gamma(g) \circ \nu_{q}$, respectively.

A morphism, $(\Gamma,\Kappa,\eta,\nu) \rightarrow (\Gamma',\Kappa',\eta',\nu')$ in $\cat{Fut}(P,Q)$ consists of an ordered pair of natural transformations, $(\alpha,\beta)$, where $\alpha : \Gamma \Rightarrow \Gamma'$ and $\beta : \Kappa \Rightarrow \Kappa'$, such that the diagrams
\begin{equation}\label{eq:morph-cond}
	\begin{tikzcd}
		& \Iota_{P} \arrow[dl, Rightarrow, "\eta"']
		\arrow[dr, Rightarrow, "\eta'"]  \\
		\Kappa\Gamma \arrow[rr, Rightarrow, "\beta*\alpha"'] & & \Kappa'\Gamma'
	\end{tikzcd}
	\quad\text{and} \quad
	\begin{tikzcd}
		& \Iota_{Q} \arrow[dl, Rightarrow, "\nu"'] \arrow[dr, Rightarrow, "\nu'"] \\
		\Gamma\Kappa \arrow[rr, Rightarrow, "\alpha*\beta"'] & & \Gamma'\Kappa'
	\end{tikzcd}
\end{equation}
commute.

Composition is component-wise, so that if 
\[
	(\alpha',\beta') : (\Gamma',\Kappa',\eta',\nu')
		\rightarrow (\Gamma'',\Kappa'',\eta'',\nu'')
\]
is another morphism, we set
\[
	(\alpha',\beta')(\alpha,\beta) = (\alpha'\circ\alpha,\beta'\circ\beta).
\]
By the interchange law of horizontal and vertical composition, this new ordered pair satisfies \eqref{eq:morph-cond}.  
It is then immediate that composition is associative, and that the identity morphism on $(\Gamma,\Kappa,\eta,\nu)$ is $(\Iota_{\Gamma}, \Iota_{\Kappa})$.

\subsection{The 2-category of future equivalences}\label{sec:fut-2-cat}

In this section we consider a bicategory that we will call $\cat{Fut}$.  
The $0$-cells of $\cat{Fut}$ are all small categories.
Then $1$-cells and $2$-cells are defined by letting $\cat{Fut}(P,Q)$ be the category described above.
Horizontal composition
\[
	\cat{Fut}(Q,R) \times \cat{Fut}(P,Q) \rightarrow \cat{Fut}(P,R)
\]
is defined by composing future equivalences, as in~\cite[Section 2.1]{grandis:2005}.
Explicitly, if $(\Gamma,\Kappa,\eta,\nu)$ and $(\Lambda, \Mu, \sigma, \tau)$ are future equivalences between $P$ and $Q$ and $Q$ and $R$, respectively, then
\[
	 (\Lambda, \Mu, \sigma, \tau) * (\Gamma, \Kappa, \eta, \nu) = (\Lambda\Gamma, \Kappa\Mu, (\Kappa \sigma \Gamma) \eta, (\Lambda \nu \Mu) \tau).
\]

The identity morphism at a small category $P$ is the quadruple, $(\Iota_{P}, \Iota_{P}, \iota_{\Iota_{P}}, \iota_{\Iota_{P}})$, where $\iota_{\Iota_{P}}$ is the identity natural transformation on the identity functor $\Iota_{P}$.

The unitor morphism $\cat{Fut}(P,Q) \rightarrow \cat{Fut}(P,Q)$ given by $(\Gamma,\Kappa,\eta,\nu) \mapsto (\Gamma,\Kappa,\eta,\nu) \circ (\Iota_{P},\Iota_{P},\iota_{\Iota_{P}},\iota_{\Iota_{P}})$, is clearly the identity map, as the components of the horizontal composition are, respectively,
\begin{itemize}
	\item $\Gamma \Iota_{P} = \Gamma$,
	\item $\Kappa \Iota_{P} = \Kappa$,
	\item $(\Iota_{P} \eta \Iota_{P}) \iota_{\Iota_{P}} = \eta$, and
	\item $(\Iota_{P} \nu \Iota_{P}) \iota_{\Iota_{P}} = \nu$.
\end{itemize}
For the same reason, the unitor on the post-composition side is also the identity map.

The associator morphisms are also identities, essentially because composition of functors and natural transformations is strictly associative.

In summary,

\begin{prop}
$\cat{Fut}$ is a strict 2-category.
\end{prop}

\subsection{Interleavings of functors}
\label{sec:fut-interleaving}

\begin{defn}
Let $(\Gamma,\Kappa,\eta,\nu) \in \cat{Fut}(P,Q)$.
We say that functors $F : P \rightarrow \cat{C}$ and $G : Q \rightarrow \cat{C}$ are $(\Gamma,\Kappa,\eta,\nu)$-\emph{interleaved} if there exist natural transformations
\[
	\varphi : F \Rightarrow G\Gamma \quad \text{and} \quad \psi : G \Rightarrow F\Kappa
\]
such that $\psi_{\Gamma}\varphi = F\eta$ and $\varphi_{\Kappa}\psi = G\nu$.
\end{defn}

\section{From future equivalences to embedding pairs}
\label{sec:fut-to-emb}

Our goal in this section is to compare embeddings and future equivalences.
Specifically, for any pair of small categories $P$ and $Q$, we construct a full and faithful functor $\cat{Fut}(P,Q)^{op} \rightarrow \cat{Emb}(P,Q)$.
These functors assemble to form an oplax functor of bicategories, $\cat{Fut} \rightarrow \cat{Emb}$.

\subsection{A functor of categories}\label{sec:Phi-functor}

In this section, we define a full and faithful contravariant functor, $\Phi : \cat{Fut}(P,Q) \rightarrow \cat{Emb}(P,Q)$.

\subsubsection{On Objects}
Given $(\Gamma,\Kappa,\eta, \nu) \in \cat{Fut}(P,Q)$, define $I_{\Gamma,\Kappa}$ to be the category with the same objects as $P \amalg Q$, such that the inclusion $P \amalg Q \rightarrow I_{\Gamma,\Kappa}$ is an embedding.
We set, for $p \in P$ and $q \in Q$,
\[
	I_{\Gamma,\Kappa}(p,q) \cong Q(\Gamma p , q)
	\quad \text{and} \quad
	I_{\Gamma,\Kappa}(q,p) \cong P(\Kappa q , p).
\]

\begin{notn}
	If $f : p \rightarrow q$ is an arrow in $I_{\Gamma,\Kappa}(p,q)$, we denote by $\bar{f} : \Gamma(p) \rightarrow q$ the corresponding arrow in $I_{\Gamma,\Kappa}(\Gamma(p),q)$, and similarly for arrows $g : q \rightarrow p$.
\end{notn}

We now define composition in $I_{\Gamma,\Kappa}$.
Since $P \amalg Q \hookrightarrow I_{\Gamma,\Kappa}$, we need only establish composition involving cross-overs from $Q$ to $P$, or vice-versa.

\begin{itemize}
	\item If $f :  p \rightarrow q$ and $g : q \rightarrow q'$, then $g \circ f$ is defined to be the unique arrow having for $\overline{g \circ f}$ the composite
	\[
		\Gamma p \xrightarrow{\bar{f}} q \xrightarrow{g} q'.
	\]
	\item If $f : p \rightarrow p'$ and $g :  p' \rightarrow q$, then $\overline{g \circ f}$ is defined to be the composite
	\[
		\Gamma p \xrightarrow{\Gamma f} \Gamma p' \xrightarrow{\bar{g}} q.
	\]
	\item If $f :  p \rightarrow q$ and $g :  q \rightarrow p'$, then ${g \circ f}$ is the composite
	\[
		p \xrightarrow{\eta_{p}} \Kappa \Gamma p \xrightarrow{\Kappa \bar{f}} \Kappa q \xrightarrow{\bar{g}} p'.
	\]
	\item The cases starting with an object in $Q$ are defined symmetrically to the above cases.
\end{itemize}

\begin{prop}\label{prop:comp-assoc}
Composition in $I_{\Gamma,\Kappa}$, as defined above, is associative and unital.
\end{prop}

\begin{proof}
The verification that composition is associative is somewhat tedious, but we will go through it for completeness (and then remove the explicit arguments later and replace them with indications of any clever bits required).

The identity maps come from the fact that $P \amalg Q \rightarrow I$ is surjective on objects.
We quickly check that the identities are neutral with respect to composition with the crossover maps.
Let $f :  p \rightarrow q$.  
Then $\overline{1_{q} \circ f} = \bar{f}$ since the composition is occurring entirely in $Q$.
On the other hand, $\overline{f \circ 1_p} = \bar{f} \Gamma(1_{p}) = \bar{f} 1_{\Gamma(p)} = \bar{f}$, using the fact that functors send identity maps to identity maps.

Now we have some straightforward checking that composition is associative when at least one crossover map is involved.
We only check the composites with source in $P$; the other verifications are symmetric.
There can be one, two, or three crossings between $P$ and $Q$.

\begin{itemize}
	\item Suppose $f :  p \rightarrow q$, $g : q \rightarrow q'$, $h : q' \rightarrow q''$.  Then 
	\[
		\overline{h \circ (g \circ f)}
		= h \circ \overline{(g \circ f)}
		= h \circ (g \circ \bar{f})
		= (h \circ g) \circ \bar{f}
		= \overline{(h \circ g) \circ f}.
	\]
	
	\item Suppose $f : p \rightarrow p'$, $g :  p' \rightarrow q$, $h : q \rightarrow q'$.
	Then
	\[
		\overline{h \circ ( g \circ f )} 
		= h \circ \overline{( g \circ f)} 
		= h \circ ( \bar{g} \circ \Gamma f )
		= (h \circ \bar{g}) \circ \Gamma f
		= \overline{(h \circ g)} \circ \Gamma f
		= \overline{(h \circ g) \circ f}.
	\]
	
	\item Suppose $f : p \rightarrow p'$, $g : p' \rightarrow p''$, $h :  p'' \rightarrow q$.
	Then 
	\[
		\overline{h \circ ( g \circ f )} 
		= \bar{h} \circ \Gamma(g \circ f)  
		= \bar{h} \circ (\Gamma g \circ  \Gamma f) 
		= (\bar{h} \circ \Gamma g) \circ \Gamma f 
		= \overline{(h \circ g)} \circ \Gamma f
		= \overline{(h \circ g) \circ f}.
	\]
	
	\item Suppose $f : p \rightarrow p'$, $g :  p' \rightarrow q$, $h :  q \rightarrow p''$.
	Then 
	\begin{eqnarray*}
		h \circ ( g \circ f) 
			& = & \bar{h} \circ \Kappa(\overline{g \circ f})  \circ \eta_{p} \\
			& = & \bar{h} \circ \Kappa(\bar{g} \circ \Gamma f) \circ \eta_{p} \\
			& = & \bar{h} \circ \Kappa\bar{g} \circ \Kappa\Gamma(f) \circ \eta_{p} \\
			& = & (\bar{h} \circ \Kappa\bar{g} \circ \eta_{p'}) \circ f \\
			& = & (h \circ g) \circ f.
	\end{eqnarray*}
	Here the trick is to use the naturality of $\eta$.
	
	\item Suppose $f:  p \rightarrow q$, $g : q \rightarrow q'$, $h :  q' \rightarrow p'$.
	Then
	\begin{eqnarray*}
		h \circ ( g \circ f)
			& = & \bar{h} \circ \Kappa(\overline{g \circ f}) \circ \eta_{p} \\
			& = & \bar{h} \circ \Kappa g \circ \Kappa \bar{f} \circ \eta_{p} \\
			& = & \overline{(h \circ g)} \circ \Kappa(f) \circ \eta_{p} \\
			& = & (h \circ g) \circ f.
	\end{eqnarray*}
	Here we just use the fact that $\Kappa$ commutes with composition.
	
	\item Suppose $f :  p \rightarrow q$, $g :  q \rightarrow p'$, $h : p' \rightarrow p''$.
	Then
	\begin{eqnarray*}
		h \circ ( g \circ f)
			& = & h \circ \bar{g} \circ \Kappa\bar{f} \circ \eta_{p}  \\
			& = & \overline{(h \circ g)} \circ \Kappa\bar{f} \circ \eta_{p}	\\
			& = & (h \circ g) \circ f.
	\end{eqnarray*}
	Here we use the fact that all morphisms are composing in $P$, where composition is associative.
	
	\item Suppose $f :  p \rightarrow q$, $g :  q \rightarrow p'$, $h :  p' \rightarrow q'$.
	Then
	\begin{eqnarray*}
		\overline{h \circ ( g \circ f )}
			& = & \bar{h} \circ \Gamma(\overline{g \circ f}) \\
			& = & \bar{h} \circ \Gamma( \bar{g} \circ \Kappa\bar{f} \circ \eta_{p} ) \\
			& = & \bar{h} \circ \Gamma\bar{g} \circ \Gamma\Kappa\bar{f} \circ \Gamma\eta_{p}  \\
			& = & \bar{h} \circ \Gamma\bar{g} \circ \Gamma\Kappa\bar{f} \circ \nu_{\Gamma p} ) \\
			& = & \bar{h} \circ \Gamma\bar{g} \circ \nu_{q} \circ \bar{f}) \\
			& = & (h \circ g) \circ \bar{f} \\
			& = & \overline{(h \circ g) \circ f}.
	\end{eqnarray*}
	Here we finally use the coherence condition, $\nu \Gamma = \Gamma \eta$, along with naturality of $\nu$.
\end{itemize}
We can also put these in terms of diagrams, I suppose.
\end{proof}

\begin{rmk}
We could condense the proof by remarking that $I_{\Gamma,\Kappa}$ is the Grothendieck construction for a certain functor into $\cat{Cat}$.
\end{rmk}  

We can now compare our two notions of interleaving (Sections \ref{sec:I-interleaving} and \ref{sec:fut-interleaving}). 
Let $P, Q$ be weighted categories and let $F:P \to \cat{C}$ and $G:Q \to \cat{C}$.
Let $(\Gamma,\Kappa,\eta, \nu) \in \cat{Fut}(P,Q)$.

% Recall from Section~\ref{sec:I-interleaving}.
% %\begin{defn}
% Let $P \rightarrow I \leftarrow Q \in \cat{Emb}(P,Q)$.
% We say that the functors $F : P \rightarrow \cat{C}$ and $G : Q \rightarrow \cat{C}$ are $I$-interleaved if $F + G : P \amalg Q \rightarrow \cat{C}$ extends to $I \rightarrow \cat{C}$.
% %\end{defn}

\begin{prop} \label{prop:fut-emb}
$F$ and $G$ are $(\Gamma,\Kappa)$-interleaved if and only if $F$ and $G$ are $I_{(\Gamma,\Kappa)}$-interleaved.
\end{prop}

\subsubsection{On Morphisms}

Let $(\alpha,\beta):(\Gamma,\Kappa,\eta,\nu) \rightarrow (\Gamma',\Kappa',\eta',\nu')$ be a morphism in $\cat{Fut}(P,Q)$.
Let $I = \Phi(\Gamma,\Kappa,\eta,\nu)$, $I' = \Phi(\Gamma', \Kappa', \eta', \nu')$.
Define a functor $\Phi(\alpha,\beta) : I' \rightarrow I$ as follows.
$\Phi(\alpha,\beta)$ is the identity on objects.
We set $\Phi(\alpha,\beta)$ to be the identity on morphisms entirely in $P$ or $Q$.
As a consequence, it preserves all identity morphisms.
On crossover morphism sets, we define $\Phi(\alpha,\beta)$ via pullback along either $\alpha$ or $\beta$.
Specifically,
\[
	I'(p,q) \cong Q(\Gamma'(p),q) \xrightarrow{\alpha^{*}_{p}}
		Q(\Gamma(p),q) \cong I(p,q)
\]
\[
	I'(q,p) \cong P(\Kappa'(q),p) \xrightarrow{\beta^{*}_{q}} 
		P(\Kappa(q),p) \cong I(q,p).
\]

\begin{prop}\label{prop:phi-fun}
The mappings above define a functor $\Phi(\alpha,\beta) : I' \rightarrow I$.
\end{prop}

\begin{proof}
Since $\Phi(\alpha,\beta)$ is the identity on morphisms in $P$ and in $Q$, we need only to verify that $\Phi(\alpha, \beta)$ preserves composition when crossover morphisms are involved, and this is straightforward.
\begin{itemize}
	\item If $f : p \rightarrow q$ and $g : q \rightarrow q'$ in $I'$, then $\overline{g \circ f} = g \circ \bar{f} \in Q(\Gamma'p,q')$.  
	Therefore 
	\begin{eqnarray*}
		\Phi(\alpha,\beta)(g \circ f) 
			& = & \alpha^{*}_{p}(g \circ \bar{f})	\\
			& = & g \circ \bar{f} \circ \alpha_{p}			\\
			& = & g \circ (\alpha^{*}_{p}(\bar{f}))	\\
			& = & \Phi(\alpha,\beta)(g) \circ \Phi(\alpha,\beta)(f).
	\end{eqnarray*}
	\item Let $f : p \rightarrow p'$ and $g : p' \rightarrow q$ in $I'$. 
	Then 
	\[
		\Phi(\alpha,\beta)(g \circ f) 
		= \bar{g} \circ (\Gamma'f \circ \alpha_{p}),
	\]
	while
	\[
		\Phi(\alpha,\beta)(g) = \bar{g} \circ \alpha_{p'}
		\quad \text{and} \quad
		\Phi(\alpha,\beta)(f) = f.
	\]
	Therefore we see that $\Phi(\alpha,\beta)(g \circ f) = \Phi(\alpha,\beta)(g) \circ \Phi(\alpha,\beta)(f)$ by considering the diagram,
	\[
		\begin{tikzcd}
			\Gamma(p) \arrow[r, "\Gamma f"]
				\arrow[d, "\alpha_{p}"']
				& \Gamma(p') \arrow[d, "\alpha_{p'}"] \\
			\Gamma'(p) \arrow[r, "\Gamma'f"']
				& \Gamma'(p') \arrow[r, "\bar{g}"'] & q
		\end{tikzcd}
	\]
	which commutes since $\alpha$ is a natural transformation.
	\item Suppose that $f : p \rightarrow q$ and $g : q \rightarrow p'$.
	Then $g \circ f = \bar{g} \circ \Kappa'(f) \circ \eta'_{p}$, which is a morphism in $P(p,p')$, so $\Phi(\alpha,\beta)$ does not change it.
	On the other hand, $\Phi(\alpha,\beta)(f) = \bar{f} \circ \alpha_{p}$ and $\Phi(\alpha,\beta)(g) = \bar{g} \circ \beta_{q}$.
	It suffices then to show that the diagram 
	\[
		\begin{tikzcd}
			P  \arrow[r, "\eta_{p}"]
				\arrow[drr, "\eta'_{p}"']
				& \Kappa\Gamma(p) \arrow[r, "\Kappa(\alpha_{p})"]
				& \Kappa\Gamma'(p) \arrow[r, "\Kappa\bar{f}"]
					\arrow[d, "\beta_{\Gamma'(p)}"]
				& \Kappa(q) \arrow[d, "\beta_{q}"] \\
			& & \Kappa'\Gamma'(p) \arrow[r, "\Kappa'\bar{f}"']
				& \Kappa'(q) \arrow[d, "\bar{g}"] \\
			& & & p'
		\end{tikzcd}
	\]
	commutes.
	The left triangle commutes because $\beta_{\Gamma'(p)} \circ \Kappa(\alpha_{p}) = (\beta * \alpha)_{p}$.
	The square commutes because $\beta$ is a natural transformation.
\end{itemize}
In conclusion, $\Phi(\alpha,\beta):I'\rightarrow I$ is a functor. 
\end{proof} 

\begin{lem}\label{lem:phi-id}
$\Phi(\Iota_{\Gamma},\Iota_{\Kappa})$ is the identity.
\end{lem}

\begin{proof}  
This is trivial, since pulling back along identity maps is simply the identity.
\end{proof}

\begin{lem}\label{lem:phi-comp}
$\Phi$ preserves composition.
\end{lem}

\begin{proof}
We need to show that if $(\alpha',\beta') : (\Gamma',\Kappa',\eta',\nu') \rightarrow (\Gamma'',\Kappa'',\eta'',\nu'')$, then $\Phi ( (\alpha',\beta')(\alpha,\beta) ) = \Phi(\alpha,\beta) \Phi(\alpha',\beta')$.
It suffices to check crossover morphisms, since everything is the identity restricted to $P$ or $Q$.
The result follows from the fact that $(\alpha'\alpha)^{*} = \alpha^{*} (\alpha')^{*}$, and similarly for $\beta$.
\end{proof}

\begin{thm}\label{thm:phi-full-faithful}
$\Phi : \cat{Fut}(P,Q) \rightarrow \cat{Emb}(P,Q)$ is a full and faithful contravariant functor.
\end{thm}

\begin{proof}
Proposition~\ref{prop:phi-fun} and Lemmas~\ref{lem:phi-id} and~\ref{lem:phi-comp} show that $\Phi$ is a contravariant functor.
To see that $\Phi$ is full, we use something along the lines of a Yoneda argument.
Let
\[
	\begin{tikzcd}
		& I' \arrow[dd, "F"] \\
		P \arrow[ur, hook] \arrow[dr, hook'] & & Q \arrow[ul, hook'] \arrow[dl, hook] \\
		& I
	\end{tikzcd}
\]
be a morphism in $\cat{Emb}(P,Q)$, where $I' = \Phi(\Gamma', \Kappa', \eta', \nu')$ and $I = \Phi(\Gamma, \Kappa, \eta, \nu)$.
We define a natural transformation $\alpha : \Gamma \Rightarrow \Gamma'$ as follows.
Let $p \in P$. 
Then $\alpha_{p} : \Gamma(p) \rightarrow \Gamma'(p)$ is the image of the identity under the mapping
\[
	Q(\Gamma'(p),\Gamma'(p)) \cong I'(p,\Gamma'(p)) \xrightarrow{F} I(p,\Gamma'(p)) \cong Q(\Gamma(p),\Gamma'(p)).
\]
Let $i_{p} : p \rightarrow \Gamma'p$ be the unique arrow such that $\bar{i}_{p} = 1_{\Gamma'p}$. 
Then ${\alpha}_{p} = \overline{F(\iota_{p})}$. 
We claim that the $\alpha_{p}$ are the components of a natural transformation.
Suppose $\varphi : p_{1} \rightarrow p_{2}$ is a morphism in $P$.

Since $F:I' \rightarrow I$ is a functor, the diagram
\[
	\begin{tikzcd}
		I'(\Gamma'(p_{1}),\Gamma'(p_{2})) \times I'(p_{1},\Gamma'(p_{1}))
			\arrow[r, "F \times F"]
			\arrow[d]
		& I(\Gamma'(p_{1}),\Gamma'(p_{2})) \times I(p_{1},\Gamma'(p_{1}))
			\arrow[d]	\\
		I'(p_{1}, \Gamma'(p_{2}))
			\arrow[r, "F"']
		& I(p_{1}, \Gamma'(p_{2}))
	\end{tikzcd}
\]
commutes, in which the vertical arrows are composition.  
Let $\gamma : p_{1} \rightarrow \Gamma'p_{2}$ be the unique element of $I'(p_{1}, \Gamma'p_{2})$ that satisfies $\bar{\gamma} = \Gamma'\varphi$ in $Q(\Gamma'p_{1}, \Gamma'p_{2})$.
Tracking $(\Gamma'(\varphi),i_{p_{1}})$, we see that $F(\gamma) = \Gamma'(\varphi) \circ F(i_{p_{1}})$.

Similarly, the diagram
\[
	\begin{tikzcd}
		I'(p_{2},\Gamma'(p_{2})) \times I'(p_{1},p_{2})
			\arrow[r, "F \times F"]
			\arrow[d]
		& I(p_{2},\Gamma'(p_{2})) \times I(p_{1},p_{2})
			\arrow[d] \\
		I'(p_{1}, \Gamma'(p_{2}))
			\arrow[r, "F"']
		& I(p_{1}, \Gamma'(p_{2}))
	\end{tikzcd}
\]
commutes, where the vertical arrows are again composition.
This time we track $(i_{p_{2}},\varphi)$.
Since 
\[
	\overline{\iota_{p_{2}} \circ \varphi} = 1_{\Gamma'(p_{2})} \circ \Gamma'(\varphi) = \Gamma'(\varphi)
\]
by the definition of the composition in $I'$, it follows that $i_{p_{2}} \circ \varphi = \gamma$.
Therefore $F(\gamma) = F(\iota_{p_{2}}) \circ \Gamma(\varphi)$.

We now have that $\Gamma'\varphi \circ F(i_{p_{1}}) = F(i_{p_{2}}) \circ \Gamma\varphi$ in $I(p_{1},\Gamma'(p_{2}))$.  
The image of this equality in $Q(\Gamma(p_{1}),\Gamma(p_{2}))$ is $\Gamma'(\varphi) \circ \alpha_{p_{1}} = \alpha_{p_{2}} \circ \Gamma(\varphi)$; that is, the diagram
\[
	\begin{tikzcd}
		\Gamma(p_{1}) 
			\arrow[r, "\alpha_{p_{1}}"]
			\arrow[d, "\Gamma(\varphi)"']
		& \Gamma'(p_{1})
			\arrow[d, "\Gamma'(\varphi)"]
			\\
		\Gamma(p_{2})
			\arrow[r, "\alpha_{p_{2}}"']
		& \Gamma'(p_{2})
	\end{tikzcd}
\]
commutes, and so the $\alpha_{p}$ assemble to form a natural transformation, $\alpha : \Gamma \Rightarrow \Gamma'$.

Similarly, we use $F$ to construct $\beta : \Kappa \Rightarrow \Kappa'$.

Now we show that the triangles (\ref{eq:morph-cond}) commute.  
Let us take a moment to recall the horizontal composition of natural transformations:  to define the component $(\alpha * \beta)_{q} : \Gamma \Kappa(q) \rightarrow \Gamma'\Kappa'(q)$, we compose along either path in the commuting square,
\[
	\begin{tikzcd}
		\Gamma \Kappa(q)
			\arrow[r, "\Gamma(\beta_{q})"]
			\arrow[d, "\alpha_{\Kappa(q)}"']
		& \Gamma \Kappa'(q)
			\arrow[d, "\alpha_{\Kappa'(q)}"]
			\\
		\Gamma'\Kappa(q)
			\arrow[r, "\Gamma'(\beta_{q})"']
		& \Gamma'\Kappa'(q).
	\end{tikzcd}
\]
For all $q \in Q \subset I'$, let $j_{q} : q \rightarrow \Kappa'(q)$ be the unique arrow that satisfies $\bar{j}_{q} = 1_{\Kappa'(q)}$ in $P(\Kappa'(q),\Kappa'(q))$.
We remark that, from the definitions,
\begin{equation}\label{eq:nu-factors}
	i_{\Kappa'(q)} \circ j_{q} = \bar{i}_{\Kappa'(q)} \circ \Gamma'(\bar{j}_{q}) \circ \nu'_{q}
	= 1_{\Gamma'\Kappa'(q)} \circ \Gamma'(1_{\Kappa'(q)}) \circ \nu'_{q} = \nu'_{q}.
\end{equation}

On the other hand,
\begin{eqnarray*}
	(\alpha * \beta)_{q} \circ \nu_{q}
		& = & \alpha_{\Kappa'(q)} \circ \Gamma(\beta_{q}) \circ \eta_{q} \\
		& = & \overline{F(i_{\Kappa'(q)})} \circ \Gamma(\overline{F(j_{q})}) \circ \eta_{q} \\
		& = & F(i_{\Kappa'(q)}) \circ F(j_{q}) \\
		& = & F( \eta_{q} ) \quad \text{by (\ref{eq:nu-factors})} \\
		& = & \eta_{q}
\end{eqnarray*}
since $\eta_{q}$ is not a crossover morphism.

The commutativity for the other triangle follows by a symmetric argument.
Therefore, $F = \Phi(\alpha,\beta)$, and $\Phi$ is full.

Suppose now that $\Phi(\alpha_{1},\beta_{1}) = \Phi(\alpha_{2},\beta_{2})$.
Then for all $p \in P$ and all $q \in Q$, and all $f : p \rightarrow q \in I'$, we have that $\alpha_{1}^{*}(f) = \alpha_{2}^{*}(f)$.
In particular, if we take $q = \Gamma'(p)$ and $f = i_{p} : p \rightarrow \Gamma'(p)$, then $\bar{i}_{p} = 1_{\Gamma'(p)}$, so we find that $\alpha_{1} = \alpha_{2}$.
In the same way, $\beta_{1} = \beta_{2}$, and so $\Phi$ is faithful.
\end{proof}

\subsection{A functor of bicategories} \label{sec:functor-bicategories}

In this section we define an oplax functor, $\Phi : \cat{Fut} \rightarrow \cat{Emb}$.  
It is the identity on objects, and for each $P,Q$, we already have a functor $\Phi : \cat{Fut}(P,Q)^{op} \rightarrow \cat{Emb}(P,Q)$.
Let $(\Gamma, \Kappa , \eta, \nu)$ and $(\Lambda, \Mu, \sigma, \tau)$ be horizontally composable future equivalences, which we picture as:
\[
	\begin{tikzcd}
		P \arrow[r, bend left, "\Gamma"]
		& Q \arrow[l, bend left, "\Kappa"]
			\arrow[r, bend left, "\Lambda"]
		& R. \arrow[l, bend left, "\Mu"]
	\end{tikzcd}
\]

\begin{prop}\label{prop:Phi-homo}
There is a natural embedding
\[
	 \Phi((\Lambda,\Mu) * (\Gamma,\Kappa)) \hookrightarrow \Phi(\Lambda,\Mu) \circ \Phi(\Gamma,\Kappa)
\]
that makes the diagram
\[
	\begin{tikzcd}
		& \Phi((\Lambda,\Mu) * (\Gamma,\Kappa))
			\arrow[dd, hook] \\
		P
			\arrow[ur, hook] 
			\arrow[dr, hook]
			& & R
				\arrow[ul, hook']
				\arrow[dl, hook'] \\
		&\Phi(\Lambda,\Mu) \circ \Phi(\Gamma,\Kappa)
	\end{tikzcd}
\]
commute.
\end{prop}

\begin{proof}
Let $J = \Phi(\Lambda,\Mu,\sigma,\tau)$, $I = \Phi(\Gamma,\Kappa , \eta, \nu)$, and $L = \Phi((\Lambda,\Mu,\sigma,\tau) * (\Gamma, \Kappa, \eta, \nu))$.
If $p \in P$ and $r \in R$, then
\begin{eqnarray*}
	(J \circ I)(p,r)
		& = & \coprod_{q \in Q} J(q,r) \times I(p,q) / \sim	\\
		& \cong & \coprod_{q \in Q} R(\Lambda(q) , r) \times  Q(\Gamma(p),q) / \sim \\
\end{eqnarray*}
where $\sim$ is the equivalence relation of Proposition~\ref{prop:pushout-reln}.
By definition, 
\[
	L(p,r) \cong R(\Lambda\Gamma(p),r).
\]
For each $q \in Q$, define a mapping $\chi_{q} : J(q,r) \times I(p,q) \rightarrow L(p,r)$ by
\[
	\overline{\chi_{q}( q \xrightarrow{g} r , p \xrightarrow{f} q )} = \bar{g} \circ \Lambda(\bar{f})
\]
If the diagram
\[
	\begin{tikzcd}
		p \arrow[r, "f"] \arrow[dr, "f'"']
			& q \arrow[dr, "g"] \arrow[d, "h"] \\
			& q' \arrow[r, "g'"'] & r
	\end{tikzcd}
\]
commutes in $J \circ I$, then
\[
	\overline{\chi_{q'}(g',f')} 
		= \bar{g}' \circ \Lambda(\bar{f}') 
		= \bar{g}' \circ \Lambda(h \circ \bar{f}) 
		= \bar{g}' \circ \Lambda(h) \circ \Lambda(\bar{f}) 
		= \bar{g} \circ \Lambda(\bar{f}) 
		= \overline{\chi_{q}(g,f)}
\]
so $\amalg_{q} \chi_{q}$ factors through $\sim$ to define
\[
	\chi : (J \circ I)(p,r) \rightarrow L(p,r).
\]
Suppose $p \xrightarrow{f} r \in L(p,r) \cong R(\Lambda \Gamma(p),r)$.
Let $i_{p} : p \rightarrow \Gamma(p)$ be the unique morphism in $I(p,\Gamma(p))$ such that $\bar{i}_{p} = 1_{\Gamma(p)}$.
Denote by $\bar{f} : \Lambda\Gamma(p) \rightarrow r$ the image of $f$ in $R(\Lambda \Gamma(p), r)$.  
Let $\hat{f} : \Gamma(p) \rightarrow r$ the corresponding image of $\bar{f}$ under the bijection $R(\Lambda \Gamma(p),r) \cong J(\Gamma(p),r)$.
Define $\xi : L(p,r) \rightarrow (J \circ I)(p,r)$ by 
\[
	\xi(f) = [ \hat{f} , i_{p} ].
\]
We claim that $\xi$ is the inverse of $\chi$.
Indeed,
\[
	\overline{\chi_{\Gamma(p)}(\hat{f},i_{p})}
	= \bar{f} \circ \Lambda(\bar{i}_{p})
	= \bar{f} \circ \Lambda(1_{\Gamma{p}})
	= \bar{f}
\]
so $\chi\xi = 1$.

Next, let $[q \xrightarrow{g} r, p \xrightarrow{f} q] \in (J \circ I)(p,r)$.
The diagram
\[
	\begin{tikzcd}
		p \arrow[r, "i_{p}"]
			\arrow[dr, "f"']
			& \Gamma(p) \arrow[dr, "g \circ \bar{f}"]
				\arrow[d, "\bar{f}"] \\
			& q \arrow[r, "g"]
			& r
	\end{tikzcd}
\]
commutes in $J \circ I$.
It follows that if we chase the image of $\chi[g,f] \in L(p,r)$ through $R(\Lambda \Gamma(p), r)$ into $J(\Gamma(p),r)$, we get $(\chi[g,f])^{\wedge} = g \circ \bar{f}$.
Therefore $\xi(\chi[g,f]) = [g \circ \bar{f}, i_{p}] = [g,f]$ again by the above diagram. 

Now we will check that $\chi$ is natural. 
Suppose we have  morphisms of future equivalences 
\[
	(\alpha,\beta) : (\Gamma_{1},\Kappa_{1}, \eta_{1}, \nu_{1} )
	\rightarrow (\Gamma_{2},\Kappa_{2}, \eta_{2}, \nu_{2} )
\]
in $\cat{Fut}(P,Q)$, and
\[
	(\gamma,\delta) : (\Lambda_{1},\Mu_{1}, \sigma_{1}, \tau_{1} ) \rightarrow (\Lambda_{2},\Mu_{2}, \sigma_{2}, \tau_{2} )
\]
in $\cat{Fut}(Q,R)$ (Section~\ref{sec:futpq}).
For $i = 1,2$, set $I_{i} = \Phi(\Gamma_{i}, \Kappa_{i}, \eta_{i}, \nu_{i} )$, $J_{i} = \Phi(\Lambda_{i},\Mu_{i}, \sigma_{i}, \tau_{i} )$, $L_{i} = \Phi((\Lambda_{i},\Mu_{i}, \sigma_{i}, \tau_{i} ) \circ (\Gamma_{i}, \Kappa_{i}, \eta_{i}, \nu_{i} ))$.
Since $P$ and $R$ embed in both $J_{i} \circ I_{i}$ and $L_{i}$, it suffices to show that the diagram
\[
	\begin{tikzcd}
		J_{2} \circ I_{2}(p,r) 
			\arrow[r, "\chi"] 
			\arrow[d, "{(\gamma,\delta)^{*}\circ (\alpha,\beta)^{*}}"']
		& L_{2}(p,r) 
			\arrow[d, "{(\gamma*\alpha ,
			\delta*\beta)^{*}}"] 
		\\
		J_{1} \circ I_{1}(p,r) 
			\arrow[r, "\chi"'] 
			& L_{1}(p,r)
	\end{tikzcd}
\]
commutes for $p \in P$, $r \in R$ (and of course conversely, but that case is symmetric).
Here, $(\gamma * \alpha, \delta * \beta)$ is the horizontal composition of $(\gamma, \delta)$ and $(\alpha,\beta)$.

Suppose $(\bar{g} :\Lambda_{2}(q) \rightarrow r, \bar{f} : \Gamma_{2}(p) \rightarrow q)$ represents a morphism in $(J_{2} \circ I_{2}) (p,r)$.
Consider the diagram,
\[
	\begin{tikzcd}
			& \Lambda_{1} \Gamma_{2}(p)
				\arrow[r, "\Lambda_{1}(\bar{f})"]
				\arrow[dd, "\gamma_{\Gamma_{2}(p)}"]
				& \Lambda_{1}(q)
					\arrow[dd, "\gamma_{q}"']
					\arrow[dr, "\gamma^{*}(\bar{g})"]	\\
		\Lambda_{1} \Gamma_{1} (p)
			\arrow[dr, "(\gamma*\alpha)_{p}"']
			\arrow[ur, "\Lambda_{1}(\alpha_{p})"]
			& & & r \\	
		& \Lambda_{2} \Gamma_{2} (p)
			\arrow[r, "\Lambda_{2}(\bar{f})"]
		& \Lambda_{2}(q)
			\arrow[ur, "\bar{g}"]
	\end{tikzcd}
\]
The top path represents  $(\gamma*\alpha, \delta*\beta)^{*}\chi(f,g)$, while the bottom path is $\chi(g,\alpha^{*}_{p}(\bar{f}))$.
The left triangle commutes by definition of horizontal composition.
The right triangle commutes by definition of pullback morphism.
The square commutes because $\gamma$ is a natural transformation.
\end{proof}

\section{Weighted future equivalences} \label{sec:wfut}

In this section we define the category $\cat{wFut}(P,Q)$ of future equivalences between weighted categories $P$ and $Q$.
We show that $\cat{wFut}(P,Q)$ is a weighted set, and that these weighted sets piece together properly to make the 1-skeleton of a bicategory $\cat{wFut}$ into a weighted category.
Finally, we define the interleaving distance between functors $F : P \rightarrow \cat{C}$ and $G: Q \rightarrow \cat{C}$, when $P$ and $Q$ are small weighted categories.

\subsection{Future equivalences of weighted categories}

Let $\cat{wFut}(P,Q)$ be the category of future equivalences between the weighted categories $P$ and $Q$, which consist of interleavings $(\Gamma, \Kappa, \eta, \nu)$, where $\Gamma$ and $\Kappa$ are non-expansive.

Let $\cat{wFut}$ have small weighted categories as $0$-cells, and hom-categories $\cat{wFut}(P,Q)$. 
By the same observations as in Section~\ref{sec:fut-2-cat}, $\cat{wFut}$ is a 2-category.

For $(\Gamma,\Kappa,\eta,\nu) \in \cat{wFut}(P,Q)$,
the natural transformations $\eta$ and $\nu$ have induced weights via the sup norm.  
Specifically,
\[
	W(\eta) = \sup_{p \in P}\{ w(\eta_{p}:p \rightarrow \Kappa \Gamma p\}
\]
and similarly for $\nu$.  
We set 
\[
	\omega(\Gamma,\Kappa,\eta,\nu) = \frac{1}{2} \max \{ W(\eta), W(\nu) \}.
\]
The function $ \omega : \cat{wFut}(P,Q)_{0} \rightarrow [0,\infty]$ makes $\cat{wFut}(P,Q)_{0}$ into a weighted set.

% Let $\cat{wInt}(P,Q)$ be the category of cospan embeddings of $P$ and $Q$ in $\cat{wCat}$.
% Let $P \rightarrow I \leftarrow Q \in \cat{wInt}(P,Q)$. 
% Set
% \[
% h(p) = \inf\{ w(f:p \rightarrow p') \mid f \text{ factors through }Q \}
% \]
% and similarly define $h(q)$.
% Then we set
% \[
% w(P \rightarrow I \leftarrow Q) = \max
% \{ \sup_{p \in P}  h(p) , \sup_{q \in Q} h(q) \}.
% \]
% The function $w$ makes $\cat{wInt}(P,Q)$ into a weighted set.

\begin{prop}
  The functions $\omega: \cat{wFut}(P,Q) \to [0,\infty]$ determine a weighted category structure on the 1-skeleton of $\cat{wFut}$.
\end{prop}

\begin{proof}
The identity future equivalence for any object consists entirely of identity morphisms, and therefore has weight zero.

Recall (Section~\ref{sec:fut-2-cat}) that the composite of ``morphisms'' $(\Gamma, \Kappa, \eta, \nu)$ and $(\Lambda, \Mu, \sigma, \tau)$ is $(\Lambda\Gamma, \Kappa\Mu, (\Kappa\sigma\Gamma)\eta, (\Lambda\nu\Mu)\tau)$.

For all $p \in P$, $w(\Kappa\sigma\Gamma)_{p} = w(\Kappa(\sigma_{\Gamma(p)})) \leq w(\sigma_{\Gamma(p)}) \leq W(\sigma)$, since functors are non-expansive.  
Therefore
\[
	w((\Kappa\sigma\Gamma)_{p}\eta_{p}) 
	\leq W(\sigma) + w(\eta_{p}) \leq W(\sigma) + W(\eta).
\]
It follows that $W((\Kappa\sigma\Gamma)\eta) \leq W(\sigma) + W(\eta)$.

Similarly, $W((\Lambda\nu\Mu)\tau) \leq W(\nu) + W(\tau)$. 
We conclude that $\omega$ is sub-additive on composites, and so $\cat{wFut}$ is a weighted category.
\end{proof}

This weighted category has an associated Lawvere metric space. Call the corresponding Lawvere metric the \emph{future equivalence distance}. That is,

\begin{defn}
  Let $P,Q$ be weighted categories. The \emph{future equivalence distance} between $P$ and $Q$ is given by 
  \begin{equation*}
    d_{Fut}(P,Q) = \inf \{
 w(\Gamma,\Kappa,\mu,\nu) \ | \
(\Gamma,\Kappa,\mu,\nu) \in \cat{wFut}(P,Q) \}
  \end{equation*}
\end{defn}

\begin{thm} \label{thm:fut-lawvere}
  The class of weighted categories with the future equivalence distance is a Lawvere metric space.
\end{thm}

\begin{prop}
The functor $\Phi : \cat{Fut}(P,Q) \rightarrow \cat{Emb}(P,Q)$ of Section~\ref{sec:Phi-functor} restricts to define a weight-preserving functor $\Phi : \cat{wFut}(P,Q) \rightarrow (\cat{wEmb}(P,Q), \tilde{w}_{GH})$.
\end{prop}

\begin{proof}
	Given $(\Gamma,\Kappa, \eta, \nu) \in \cat{Fut}(P,Q)$, we need to show that $I_{\Gamma,\Kappa}$ is a weighted category in such a way that $P \rightarrow I_{\Gamma,\Kappa} \leftarrow Q$ becomes a diagram in $\cat{wCat}$.
	
	Set $\omega = \omega(\Gamma,\Kappa,\eta,\nu)$.
	Define weight functions by $w_{I}|_{P} = w_{P}$ and $w_{I}|_{Q} = w_{Q}$.
	For $f : p \rightarrow q$, recall that $\bar{f} : \Gamma(p) \rightarrow q$ is the image of $f$ in $Q(\Gamma(p),q)$.
	Similarly, $\bar{g} : \Kappa(q) \rightarrow p$ is the image of $g : q \rightarrow p'$ in $P(\Kappa(q),p)$.
	We set
	\[
		w_{I}(f) = w_{Q}(\bar{f}) + \omega
	\]
	and 
	\[
		w_{I}(g) = w_{P}(\bar{g}) + \omega.
	\]
	We now show that composition is sub-additive with respect to weights.
	If $h : p' \rightarrow p$, then 
	\begin{eqnarray*}
		w_{I}(f \circ h) 
			& = & w_{Q}(\bar{f} \Gamma(h))	\\
			& \leq & w_{Q}(\bar{f}) + w_{Q}(\Gamma(h)) \\
			& \leq & w_{Q}(\bar{f}) + w_{P}(h) \\
			& \leq & w_{I}(f) + w_{I}(h)
	\end{eqnarray*}
	since $\omega \geq 0$.
	Similarly, if $j : q \rightarrow q'$, $k : q' \rightarrow q$, and $\ell : p' \rightarrow p''$, then $w_{I}$ is subadditive on the composites $j \circ f$, $g \circ k$, and $\ell \circ g$.
	
	The only case that remains (up to symmetry) is the composite $g \circ f$.
	\begin{eqnarray*}
		w_{I}(g \circ f) 
			& = & w_{P}(\bar{g}\Kappa(f)\eta_{p}) \\
			& \leq & w_{P}(\bar{g}) + w_{P}(\Kappa(f)) + w_{P}(\eta_{p}) \\
			& \leq & w_{P}(\bar{g}) + w_{Q}(f) + 2\omega \\
			& = & w_{I}(g) + w_{I}(f)
	\end{eqnarray*}
	as desired.
\end{proof}

\begin{thm} \label{thm:gh-vs-fut}
  $d_{GH}(P,Q) \leq \tilde{d}_{GH}(P,Q) \leq d_{Fut}(P,Q)$.
\end{thm}

\subsection{Interleaving distance}

Weighted future equivalences give us the following distance between functors from weighted categories to the same codomain.

\begin{defn}
  Let $C$ be a category. Let $P,Q \in \cat{wCat}$. Let $F:P \to C$ and $G:Q \to C$.
We define the \emph{future equivalence distance} between $F$ and $G$ by
  \[d_{Fut}(F,G) = \inf \{ w(\Gamma,\Kappa,\mu,\nu) \ | \ (\Gamma,\Kappa,\mu,\nu) \in \cat{wFut}(P,Q), F, G \text{ are } (\Gamma,\Kappa)\text{-interleaved} \}
\]
\end{defn}

\begin{thm} \label{thm:fut-functor-lawvere}
  The class of functors on weighted categories with common codomain with the future equivalence distance is a Lawvere metric space.
\end{thm}

\begin{thm} \label{thm:interleaving-vs-fut}
  $d(F,G) \leq \tilde{d}(F,G) \leq d_{Fut}(F,G)$.
\end{thm}

\begin{thm}
  Let $F:P \to \cat{C}$ and $G:Q \to \cat{C}$. Then
  \begin{equation*}
     d_{Fut}(P,Q) \leq d_{Fut}(F,G).
  \end{equation*}
\end{thm}

\begin{thm}[Stability] 
  Let $F:P \to \cat{C}$, $G:Q \to \cat{C}$ and $H:\cat{C} \to \cat{D}$. 
  Then
  \begin{equation*}
    d(HF,HG) \leq d(F,G).
  \end{equation*}
\end{thm}

\section{Application} \label{sec:applications}

% \pb{I don't remember what I had in mind for the first two.}

% \begin{itemize}
% 	\item distances between orientable surfaces in $\mathbf{R}^{3}$ (via Gauss map) 
% 	\item directed homotopy theory
% 	\item dynamical systems (shift equivalence)
%         \item interleaving distance for Reeb graphs / Merge trees
	
% \end{itemize}
% This final section explores two applications.
% In the first, 
Here we show that a shift equivalence between discrete dynamical systems corresponds precisely to an interleaving between the corresponding functors defining the systems.
% In the second, we show how our approach can be applied to merge trees and Reeb graphs.

\subsection{Shift Equivalences of Dynamical Systems} \label{sec:dynamical}
	
	A \emph{discrete dynamical system} is a topological space $X$ along with a continuous self-map $f : X \rightarrow X$.
	From our categorical point of view, we consider a dynamical system to be a functor $F : N \rightarrow \cat{Top}$, where $N$ is the category with one object $x$ and morphisms $\varphi^{k}$ for $k \geq 0$,  $F(x)=X$ and $F(\varphi)=f$.
	
	\begin{defn}
	Dynamical systems $f : X \rightarrow X$ and $g : Y \rightarrow Y$ are said to be \emph{shift equivalent with lag} $\ell$ if there exist continuous maps $\alpha : X \rightarrow Y$ and $\beta : Y \rightarrow X$ such that $\alpha f = g \alpha$, $\beta g = f \beta$, $\beta \alpha = f^{\ell}$, and $\alpha \beta = g^{\ell}$.
	\end{defn}
	
	It turns out that shift equivalence corresponds precisely to $\cat{Fut}(N,N)$-interleavings of topological spaces.
	Since $N$ has only one object and the morphisms are free on one generator, we can completely characterize its future equivalences.
	
	\begin{prop}
		Every future equivalence from $N$ to itself is of the form $(\Iota_{N},\Iota_{N},\varphi^{\ell},\varphi^{\ell})$ for some $\ell \geq 0$.
	\end{prop}
	
	\begin{proof}
	For the sake of clarity, let $N'$ be an isomorphic copy of $N$, with single object $y$ and generating morphism $\psi$.
	
	Suppose $(\Gamma,\Kappa,\eta,\nu)$ is a future equivalence between $N$ and $N'$.
	Necessarily, $\Gamma(x)=y$ and $\Gamma(\varphi)=\psi^{i}$ for some $i \geq 0$, while $\Kappa(y)=x$ and $\Kappa(\psi)=\varphi^{j}$ for some $j \geq 0$.
	Since $\Kappa\Gamma(x)=x$, $\eta_{x} = \varphi^{k}$ for some $k \geq 0$, and similarly $\nu_{y}=\psi^{\ell}$ for some $\ell \geq 0$.
	
	Since $\eta$ is a natural transformation and $\Kappa\Gamma(x)=x$, the diagram 
	\[
		\begin{tikzcd}
			x \arrow[r, "\eta_{x}"] \arrow[d, "\varphi"']
				& x \arrow[d, "\Kappa\Gamma(\varphi)"]	\\
			x \ar[r, "\eta_{x}"'] & x
		\end{tikzcd}
	\]
	must commute.  
	Now, $\Kappa(\Gamma(\varphi)) = \Kappa(\psi^{i})=\varphi^{ij}$, so the commutativity of the above diagram implies that $ij+k=1+k$.
	It follows that $ij=1$, so $i=j=1$.  
	In other words, $\Gamma$ and $\Kappa$ are the functors defining the canonical isomorphism $N \cong N'$.
	
	The coherence condition for a future equivalence demands in particular that $\Gamma\eta = \nu\Gamma$.
	This translates in our case to $\Gamma(\varphi^{k}) = (\Gamma(\varphi))^{\ell}$; that is, $\psi^{k}=\psi^{\ell}$.
	Since the morphisms of $N$ are freely generated by $\psi$, it follows that $k = \ell$.
	\end{proof}
	
	\begin{prop}\label{prop:shift-equiv=interleaved}
	Let $F_{\ell} = (\Gamma,\Kappa,\varphi^{\ell},\psi^{\ell})$.
	An $F_{\ell}$-interleaving between dynamical systems $(X,f)$ and $(Y,g)$ is precisely a shift equivalence of lag $\ell$.
	\end{prop}
	
	We could prove Proposition~\ref{prop:shift-equiv=interleaved} directly, but it is more transparent once we see the form taken by the cospan embedding $I_{\ell}$ corresponding to $F_{\ell}$.
	Recall that $I_{\ell,0} = P_{0} \amalg Q_{0} = \{x,y\}$.
	Furthermore, $I_{\ell}(x,x)=P(x,x)=\{ \varphi^{k} \mid k \geq 0 \}$ and $I_{\ell}(y,y) = Q(y,y) = \{ \psi^{k} \mid k \geq 0 \}$.
	The morphism set $I_{\ell}(x,y) \cong Q(\Gamma(x),y) = Q(y,y)$;  we denote by $a_{i}$ the morphism corresponding to $\psi^{i}$.
	Similarly, $I_{\ell}(y,x) = \{ b_{k} \mid k \geq 0 \}$, where $b_{k}$ corresponds to $\varphi^{k} \in P(x,x)$.
	
	\begin{lem}
		For $i \geq 0$ and $k \geq 0$,
		\begin{enumerate}
			\item $a_{i+k} = \psi^{i} \circ a_{k} = a_{k} \circ \varphi^{i}$. 
			\item $b_{i+k} = b_{k} \circ \psi^{i} = \varphi^{i} \circ b_{k}$.
		\end{enumerate} 
		In particular, $a_{i} = \psi^{i} \circ a_{0} = a_{0} \circ \varphi^{i}$ and $b_{i} = b_{0} \circ \psi^{i} = \varphi^{i} \circ b_{0}$.
	\end{lem}
	
	\begin{proof}
		 The composite $\psi^{i} \circ a_{k}$ is, by definition of $I_{\ell}$, the composite in $Q$ of the arrows
		 \[
			 y = \Gamma(x) \xrightarrow{\psi^{k}} y \xrightarrow{\psi^{i}} y
		\]
		that is, $\psi^{i+k}$, which corresponds under bijection to $a_{i+k}$.
		
		By definition, the composite
		\[
			x \xrightarrow{\varphi^{i}} x \xrightarrow{a_{k}} y
		\]
		is the composite
		\[
			y = \Gamma(x) \xrightarrow{\Gamma(\varphi^{i})} \Gamma(x) = y
			\xrightarrow{\psi^{k}} y.
		\]
		Since $\Gamma(\varphi)=\psi$, we have that $a_{k} \circ \varphi^{i}$ corresponds to $\psi^{k+i}$, and so equals $a_{k+i}$.
		
		The proof of the statements concerning $b_{j}$ is similar.
	\end{proof}
	
	Let us now simplify notation by setting $a = a_{0}$ and $b = b_{0}$.
	To compute the composites of all ``crossover'' arrows between $x$ and $y$, it suffices by the above Lemma to compute the composites $ab$ and $ba$.
	
	\begin{lem}
		In $I_{\ell}$, $ba = \varphi^{\ell}$ and $ab = \psi^{\ell}$.
	\end{lem}
	
	\begin{proof}
		From the definition of composition, the composite
		\[
			x \xrightarrow{a} y \xrightarrow{b} x
		\]
		is the composite in $P$ of
		\[
			x \xrightarrow{\eta_{x}}
			\Kappa\Gamma(x) \xrightarrow{\Kappa(\psi^{0})}
			\Kappa(y) \xrightarrow{\varphi^{0}} x.
		\]
		Since $\eta_{x} = \varphi^{\ell}$, $\varphi^{0} = 1_{x}$ and $\psi^{0} = 1_{y}$, the result follows.
		
		The corresponding result for $ab$ is similar.
	\end{proof}
	
	As a result, $I_{\ell}$ is the category generated by the directed graph 
	\[
		\begin{tikzcd}
			x \arrow[r, bend left, "a"]
			\arrow[loop left, "\varphi"] 
			& y \arrow[l, bend left, "b"]
			\arrow[loop right, "\psi"]
		\end{tikzcd}
	\]
	
	subject to the relations $ab = \psi^{\ell}$, $ba = \varphi^{\ell}$, $\psi a = a \varphi$, $\varphi b = b \psi$.

% \pb{Does this generalize to monoid actions?}

% \subsection{Reeb graphs and Merge trees} \label{sec:reeb}

% \pb{Vin, can you say something about how this theory describes interleaving distance for Reeb graphs and Merge trees?}

\appendix

\section{Enriched category theory} \label{sec:enriched-theory}

\subsection{Enriched categories}
\label{sec:enriched}

\cite{kelly:enriched}

Let $(\cat{V},\otimes,I)$ be a \emph{monoidal category}; this means that there is an associative operation on the level of sets and morphisms, $(X,Y) \mapsto X \otimes Y$, $(f,g) \mapsto f \otimes g$, that is coherent up to natural isomorphism, and for which the object $I$ is neutral. See~\cite[pp. 161-170]{maclane} for details.

A $\cat{V}$-\emph{category} $\cat{C}$ consists of a collection $\cat{C}_{0}$ of objects, and for each pair $X,Y \in \cat{C}_{0}$, a \emph{hom object} $\cat{C}(X,Y) \in \cat{V}$.  For every triple $X,Y,Z \in \cat{C}_{0}$, there is a composition morphism in $\cat{V}$,
\[
	\cat{C}(Y,Z) \otimes \cat{C}(X,Y) \xrightarrow{\kappa_{X,Y,Z}} \cat{C}(X,Z).
\]
Composition must be associative.
For every object $X \in \cat{C}_{0}$, there is a unique morphism $\eta_{X} : I \rightarrow \cat{C}(X,X)$, such that for every $Y \in \cat{C}_{0}$, the composites
\[
	I \otimes \cat{C}(Y,X) \xrightarrow{\eta_{X} \otimes 1} \cat{C}(X,X) \otimes \cat{C}(Y,X)
		\xrightarrow{\kappa_{Y,X,X}} \cat{C}(Y,X)
\]
and
\[
	\cat{C}(X,Y) \otimes I \xrightarrow{1 \otimes \eta_{X}} \cat{C}(X,Y) \otimes \cat{C}(X,X)
		\xrightarrow{\kappa_{X,X,Y}} \cat{C}(X,Y)
\]
are the canonical isomorphisms.

Let $\cat{C}$ and $\cat{D}$ be $\cat{V}$-categories.  A $\cat{V}$-\emph{functor} $F : \cat{C} \rightarrow \cat{D}$ consists of a mapping, $F_{0} : \cat{C}_{0} \rightarrow \cat{D}_{0}$, and for each pair $X,Y \in \cat{C}_{0}$, a morphism in $\cat{V}$, $F_{1} : \cat{C}(X,Y) \rightarrow \cat{D}(F(X),F(Y))$,  such that the diagrams
\[
	\begin{tikzcd}
		\cat{C}(Y,Z) \otimes \cat{C}(X,Y) 
			\arrow[r, "{F_{1} \otimes F_{1}}"]
			\arrow[d]
		& \cat{D}(F(Y),F(Z)) \otimes \cat{D}(F(X),F(Y)) 
			\arrow[d] \\
		\cat{C}(X,Z) 
			\arrow[r, "{F_{1}}"'] 
		& \cat{D}(F(X),F(Z))
	\end{tikzcd}
\]
where the vertical arrows are the composition morphisms, and
\[
	\begin{tikzcd}[row sep=tiny]
		& \cat{C}(X,X)	
			\arrow[dd, "{F_{1}}"]	\\
		I 
			\arrow[ur, "{\eta_{X}}"] 
			\arrow[dr, "{\eta_{F(X)}}"']	\\
		& \cat{D}(F(X),F(X))
	\end{tikzcd}
\]
commute.
There is a category $\cat{V}$-$\cat{Cat}$ of all small $\cat{V}$-categories and $\cat{V}$-functors.

A category can be viewed as a category enriched over $\cat{Set}$, the category of sets, and a strict 2-category can be defined as a category enriched over $\cat{Cat}$, the category of small categories. In each of these cases the monoidal structure is given by finite products. A 2-functor is the corresponding enriched functor between strict 2-categories.
Let $\cat{2\text{-}Cat}$ denote the category of all small strict 2-categories and 2-functors.

\subsection{Lawvere metric spaces as enriched categories}
\label{sec:lawvere-enriched}

\cite{lawvere:1973}

% Consider $[0,\infty]$, the monoidal poset whose objects are the non-negative real numbers, along with $\infty$.  The monoidal operation is addition, where $x + \infty = \infty + x = \infty$ for all $x\in [0,\infty]$.  
A \emph{Lawvere metric space} is a small category enriched in the strict monoidal poset $(([0,\infty],\geq),+,0)$.  Suppose that $\cat{C}$ is a Lawvere metric space.  We will denote $\cat{C}(X,Y)$ by $d(X,Y)$. 
We have that $0 \leq d(X,Y) \leq \infty$ for all $X,Y \in \cat{C}$, so objects can be infinitely distant from one another in a Lawvere metric space.  
The Triangle Inequality is simply the composition morphism $\kappa_{X,Y,Z}: d(Y,Z) + d(X,Y) \geq d(X,Y)$. 
The morphism $\eta_X:0 \geq d(X,X)$ implies that $d(X,X)=0$.
There is no requirement of symmetry in distance, and for $X \neq Y$ we may have $d(X,Y)=0$.
A Lawvere metric space is also called an extended quasipseudometric space.

Let $\cat{C}$ and $\cat{D}$ be Lawvere metric spaces.  A \emph{morphism of Lawvere metric spaces}, $F:\cat{C} \rightarrow \cat{D}$, is an enriched functor.  This means that for $X,Y \in \cat{C}_{0}$, we get a ``morphism'' in $([0,\infty],\geq)$, $d(X,Y) \geq d(F(X),F(Y))$, so $F$ is nonexpansive, i.e. 1-Lipschitz.  We denote by $\cat{Lawv}$ the category of all Lawvere metric spaces, and their morphisms.

%[The following should have something to do with $\{ 0 , \infty \} \subset [0,\infty]$.  Indeed, we can consider a preordered set to be a small category enriched in $\{0,\infty\}$.]

\begin{deflemma}
There is a forgetful functor $U : \cat{Lawv} \rightarrow \cat{Cat}$ given by setting $(U\cat{C})_{0} = \cat{C}_{0}$ and, for $x,y \in \cat{C}_{0}$, 
\[
	(U\cat{C})(x,y) = \left\{
		\begin{array}{cl}
			\{*\}	& \text{if }d(x,y) < \infty	\\
			\varnothing	& \text{otherwise.}
		\end{array}
		\right.
\]
\end{deflemma}

\begin{proof}
  First we check that we have a category $U\cat{C}$. The existence of a canonical composition law follows from the triangle inequality. The associativity and identity axioms follow trivially. 

  Let $F: \cat{C} \to \cat{D} \in \cat{Lawv}$. We need to show that there exists a canonical functor $UF:U\cat{C} \to U\cat{D}$. For $x \in U\cat{C}$ define $UF(x) = F(x)$. Let $* \in U\cat{C}(x,y)$. This implies that $d(x,y) < \infty$.
Since $F$ is nonexpansive, $d(Fx,Fy) \leq d(x,y) < \infty$. Thus $U\cat{C}(Fx,Fy)=\{*\}$.
So define $UF(*) = *$.
The identity and composition conditions follow trivially.
\end{proof}

\begin{rmk}
	Note that $([0,\infty],\geq,+,0)$ is in fact \emph{closed} monoidal; namely, the functor $a \mapsto a + b$ has a right adjoint, $c \mapsto c^b$, defined on objects by $c^b = \max(c-b,0)$.
	It is easy to see that $a+b \geq c$ if and only if $a \geq c^b$.
\end{rmk}

\subsection{Weighted categories and enrichments}

A \emph{weighted set} is a set $X$ along with a function $w : X \rightarrow [0,\infty]$.
A \emph{morphism of weighted sets} $f: (X,w) \rightarrow (Y,v)$ is a mapping $f : X \rightarrow Y$ of the underlying sets that satisfies $v(f(x)) \leq w(x)$ for all $x \in X$.
The category of weighted sets and non-expansive maps is denoted by $\cat{wSet}$.
The summing of weights makes $\cat{wSet}$ into a monoidal category under the Cartesian product.
The neutral weighted set $(*,0)$ consists of one element of weight zero.

A category $\cat{C}$ enriched in $\cat{wSet}$ consists of a collection of objects, and for any two objects $x$ and $y$, a weighted set $(\cat{C}(x,y), w_{xy})$.
For any triple of objects $x$, $y$, and $z$, there is a composition morphism (in $\cat{wSet}$),
\[
	(\cat{C}(y,z),w_{yz}) \times (\cat{C}(x,y),w_{xy}) \rightarrow (\cat{C}(x,z),w_{xz}).
\]
In particular, if $f : x \rightarrow y$ and $g : y \rightarrow z$, then $w_{xz}(gf) \leq w_{yz}(g) + w_{xy}(f)$.
Also, for every object $x$, there is a unit morphism of weighted sets, $(*,0) \rightarrow (\cat{C}(x,x), w_{xx})$, $* \mapsto 1_{x}$, that is neutral for composition.  
In particular, $w_{xx}(1_{x}) = 0$.
In short, $\cat{C}$ is precisely a weighted category.

\subsection{From weighted categories to Lawvere metric spaces}
\label{sec:comparison}

% \begin{proposition}
% There exists functor $\cat{Mcat} \rightarrow \cat{Lawv}$ (take inf over all paths to get distance).
% \end{proposition}

The infimum defines a morphism of monoidal categories, $\iota : \cat{wSet} \rightarrow [0,\infty]$.
On objects, $\iota(X) = \inf_{x \in X} w(x)$.  
If $f : X \rightarrow Y$ is a morphism in $\cat{wSet}$, then $w(f(x)) \leq w(x)$ for all $x \in X$, and so 
\[
	\inf_{y \in Y} (w(y)) \leq w(f(x)) \leq w(x)
\]
for all $x \in X$.
It follows by the universal property of the infimum that $\inf_{x \in X}\{ w(x) \} \geq \inf_{y \in Y} \{ w(y) \}$, i.e., $\iota(X) \geq \iota(Y)$.

A straightforward argument shows that $\iota ( X \times Y ) = \iota(X) + \iota(Y)$.

In the other direction, define $\kappa : [0,\infty] \rightarrow \cat{wSet}$ to be, on the level of objects, the function that sends $a$ to the set $\{ *_{a} \}$, where $w(*_{a}) = a$.
If $a \geq b$ then the unique map $\{ *_{a} \} \rightarrow \{ *_{b} \}$ is a morphism in $\cat{wSet}$, so $\kappa$ is a functor.  
Since there is a natural isomorphism of weighted sets $\{ *_{a} \} \times \{ *_{b} \} \cong \{ *_{a+b} \}$, $\kappa$ is monoidal.

\begin{prop}
	The monoidal functor $\kappa$ is right-adjoint to $\iota$.
\end{prop}

\begin{proof}
	Let $(X,w) \in \cat{wSet}$ and $a \in [0,\infty]$.
	If $w(x) \geq a$ for all $x \in X$, then there exists a unique arrow $(X,w) \rightarrow \kappa(a)$, otherwise $\cat{wSet}((X,w),\kappa(a)) = \varnothing$.
	In the former case holds if and only if $\inf_{x \in X} w(x) \geq a$, so we get a canonical bijection,
	\[
		\cat{wSet}((X,w),\kappa(a)) \cong [0,\infty](\iota(X,w),a).
	\]
\end{proof}

By standard enriched-category theory~\cite[Chapter 6]{borceux}, $\iota$ and $\kappa$ induce an adjunction, 
\[
	\begin{tikzcd}
		\cat{wCat} \arrow[rr, bend left, "\iota_{*}"] & \bot & \cat{Lawv}
		\arrow[ll, bend left, "\kappa_{*}"]
	\end{tikzcd}
\]
which we make explicit in the following definition/lemma.

\begin{deflemma}
There is a functor $i_{*}: \cat{wCat} \to \cat{Lawv}$ defined as follows.
For $(\cat{C},w_C) \in \cat{wCat}$, let $i_{*}(\cat{C},w_C)$ be the Lawvere metric space whose objects are those of $\cat{C}$.
For $X,Y \in \cat{C}$, define
\begin{equation*}
  d(X,Y) = \inf_{\gamma \in C(X,Y)} w_\cat{C}(\gamma),
\end{equation*}
where $d(X,Y) = \infty$ if $\cat{C}(X,Y) = \varnothing$.
\end{deflemma}

\begin{proof}
First we verify that $i_{*}(\cat{C},w_C)$ is a Lawvere metric space.
For all $X \in \cat{C}$, $w_C(1_X) = 0$, so it follows that $d(X,X) = \inf_{\gamma \in \cat{C}(X,X)} w_C(\gamma) = 0$.
For $X,Y,Z \in \cat{C}$,
\begin{eqnarray*}
d(Y,Z)+d(X,Y) &=& \inf_{\gamma' \in \cat{C}(Y,Z)} w_C(\gamma') + \inf_{\gamma \in \cat{C}(X,Y)} w_C(\gamma)\\
&\geq& \inf_{\gamma'' \in \cat{C}(X,Z)} w_C(\gamma'')\\
&=& d(X,Z).
\end{eqnarray*}

For a morphism $(F,\alpha): (\cat{C},w_C) \to (\cat{D},w_D) \in \cat{wCat}$, let $i_{*}(F,\alpha)$ be the functor from $i_{*}(\cat{C},w_C)$ to $i_{*}(\cat{D},w_D)$ given by $i_{*}(F,\alpha)(X) = F(X)$ for $X \in \cat{C}$.
We need to show that $i_{*}(F,\alpha)$ is nonexpansive. That is, for $X,Y \in \cat{C}$, $d(X,Y) \geq d(FX,FY)$. We can verify this as follows:
\begin{eqnarray*}
d(X,Y) &=& \inf_{\gamma \in \cat{C}(X,Y)} w_C(\gamma)\\
 &\geq& \inf_{\gamma\in \cat{C}(X,Y)} w_D(F(\gamma))\\
 &\geq& \inf_{\gamma' \in \cat{D}(FX,FY)} w_D(\gamma')\\
 &=& d(FX,FY),
\end{eqnarray*}
where the first inequality follows from the nonexpansiveness of $F$.
The identity and composition properties follow trivially.
\end{proof}

% \subsection{Stability}
% \label{sec:stability}

% \begin{rmk}
% Let $F,G : \cat{C} \rightarrow \cat{D}$ be functors and $\eta : F \Rightarrow G$ a natural transformation. If $H : \cat{A} \rightarrow \cat{C}$ is a functor, then $\eta_{H} : FH \Rightarrow GH$ is the natural transformation defined by $(\eta H)_{a} : \eta_{H(a)}: FH(a) \rightarrow GH(a)$ for all $a \in \cat{A}_{0}$.
% \end{rmk}

\section{Hausdorff distance as an interleaving distance}
\label{sec:hausdorff-interleaving}

In this section we show that the Hausdorff distance between subsets of a metric space (see Section~\ref{sec:hausdorff-metric} can be interpreted as an interleaving distance.

For our third definition of Hausdorff distance we need the following observation.

\begin{lem}
  If $A \subset B^r$ then for all $s \geq 0$, $A^s \subset B^{r+s}$.
\end{lem}

\begin{proof}
  Let $x \in A^s$ and let $\eps > 0$. Then there exists $a \in A$ such that $d(a,x) \leq s + \eps/2$. Since $A \subset B^r$, there exists $b \in B$ such that $d(b,a) \leq r + \eps/2$. By the triangle inequality, $d(b,x) \leq r+x+\eps$. Thus $\inf_{b\in B}d(b,x) \leq r+s$ and $x \in B^{r+s}$.
\end{proof}

Now for $a \geq 0$, let $F_A(a) = A^a$ and $F_B(a) = B^a$. Following~\cite{bdss:1}, we have the following.

\begin{defn}
  Two families of sets $F = (F(a))_{a \geq 0}$ and $G = (G(a))_{a \geq 0}$ and said to be \emph{$r$-interleaved} if for all $a \geq 0$, $F(a) \subset G(a+r)$ and $G(a) \subset F(a+r)$.
The \emph{interleaving distance} of $F$ and $G$ is given by
\begin{equation*}
  d(F,G) = \inf \{ r \geq 0 \ | \ \text{$F$ and $G$ are $r$-interleaved} \}.
\end{equation*}
\end{defn}
Combining this definition with the previous two lemmas we have the following.

\begin{cor}
  Let $A,B \subset M$. Let $F_A$ and $F_B$ be the families of offsets given by $F_A(a) = A^a$ and $F_B(a) = B^a$. Then the Hausdorff distance of $A$ and $B$ equals the interleaving distance of $F_A$ and $F_B$. That is,
  \begin{equation*}
    d_H(A,B) = d(F_A,F_B).
  \end{equation*}
\end{cor}

\begin{rmk}
  In the language of~\cite{bdss:1}, $F_A$ and $F_B$ are functors from the indexing category given by the poset $([0,\infty),\leq)$ to the category of subsets of $M$ with partial order given by inclusion. The indexing category has a (super)linear family of translations given by $\Omega_r(a) = a + r$. For details see~\cite[Sections 3.2, 3.5]{bdss:1}.
\end{rmk}

\begin{rmk}
  The Hausdorff distance between subsets of a Lawvere metric space can also be interpreted as an interleaving distance. However, since the endomorphisms ${}^s(-)$ and $(-)^s$ need not commute with each, the indexing category needs to be a continuous version of the free monoid on two generators.
\end{rmk}

% \bibliography{persistence}
% \bibliographystyle{hplain}

\end{document}